\documentclass[12pt,reqno]{amsart}
\usepackage{amsmath,amssymb,amsthm,mathtools}
\usepackage{times}
\usepackage{fullpage}
\usepackage[colorlinks=true,citecolor=blue,linkcolor=blue,urlcolor=blue]{hyperref}
\usepackage{comment}
\newtheorem{theorem}{Theorem}[section]
\newtheorem{proposition}[theorem]{Proposition}
\newtheorem{lemma}[theorem]{Lemma}
\newtheorem{corollary}[theorem]{Corollary}
\theoremstyle{remark}
\newtheorem{remark}[theorem]{Remark}
\numberwithin{equation}{section}

\newcommand{\Fp}{\mathbb F_p}
\newcommand{\Fps}{\mathbb F_p^\times}
\newcommand{\rsum}{\mathbin{\widehat{+}}}

\title{Additive decompositions of multiplicative subgroups\\via differential identities}
\author{Albert Cochrane}
\address{Department of Mathematics, Kansas State University, Manhattan, Kansas, USA}
\email{albertc@ksu.edu}
\author{Chi Hoi Yip}
\address{Department of Mathematics, Hong Kong University of Science and Technology, Clear Water Bay, Hong Kong}
\email{machyip@ust.hk}
\author{Semin Yoo}
\address{Discrete Mathematics Group, Institute for Basic Science, 55 Expo-ro Yuseong-gu, Daejeon 34126, South Korea}
\email{syoo19@ibs.re.kr}

\begin{document}

\begin{abstract} 
We develop a local-to-global differential framework for additive decomposition problems involving multiplicative subgroups of prime fields. Starting from Hanson--Petridis-type auxiliary polynomials, we use degree bounds, in the spirit of Stepanov's method, to lift local coefficient relations at their roots to global differential identities.
This yields a unified treatment of
\[
 H=A+B,\qquad H\cup\{0\}=A-A,\qquad H=A\rsum A,\qquad H\cup\{0\}=A\rsum A,
\]
where \(H\) is a proper multiplicative subgroup. This circle of problems is motivated by S\'ark\"ozy's conjecture on the additive irreducibility of nonzero quadratic residues and its generalizations to multiplicative subgroups. Rudnev and Tyrrell recently classified all decompositions \(H=A+B\), building on the approach introduced by Hanson--Petridis and further developed by Kalmynin.

Our framework gives a new polynomial proof of the Rudnev--Tyrrell classification and substantially streamlines the existing proofs: it gives an independent proof of Kalmynin's equal-size theorem and reduces the classification to direct coefficient comparisons, avoiding the residue-theoretic input and more elaborate arithmetic analysis of earlier proofs. It also yields a streamlined proof of Kalmynin's resolution of a conjecture of Lev and Sonn on \(H\cup\{0\}=A-A\). For the two restricted-sumset problems, we obtain complete classifications, substantially improving earlier results of Shkredov and Yip. We also establish some stability refinements.
\end{abstract}

\maketitle

\section{Introduction}

A recurring theme in arithmetic combinatorics is the tension between additive and multiplicative structure. Generally speaking, one does not expect a multiplicatively defined set to have a simple additive description, for example as an arithmetic progression or a nontrivial sumset. This principle appears in many different forms, for instance, in Ostmann's inverse Goldbach problem~\cite{Ostmann}; see \cite{EH15,Shao16} for the best-known partial results. It also appears in the sum--product problem of Erd\H{o}s and Szemer\'edi~\cite{ES83}; see also the recent development~\cite{BSSZ26}. In finite fields, a natural question to ask is whether proper multiplicative subgroups may admit a nontrivial additive decomposition. 

Throughout, let $p$ be an odd prime, let $\Fp$ be the finite field with $p$ elements, and write $\Fps=\Fp\setminus\{0\}$. Let $H$ be a proper multiplicative subgroup of the finite field $\Fp$. Following \cite{LS17, S12}, we use $\mathcal R_p:=\{x^2:x\in\Fps\}$ specifically for the subgroup of quadratic residues.

S\'ark\"ozy~\cite{S12} conjectured that $\mathcal R_p$ has no nontrivial additive decomposition for every sufficiently large prime $p$, that is, $\mathcal R_p$ cannot be written as a sumset $A+B=\{a+b: a\in A,b\in B\}$ for some subsets $A,B\subseteq\Fp$ with $|A|,|B|\geq2$. More generally, one can ask whether a proper multiplicative subgroup $H<\Fps$ has a nontrivial additive decomposition. For a fixed index $r\ge2$, the corresponding generalized S\'ark\"ozy conjecture predicts that, whenever $p\equiv1\pmod r$ is sufficiently large, the subgroup of index $r$ in $\Fps$ admits no such decomposition. This problem and its variants have been studied extensively by S\'ark\"ozy, Shparlinski, Shkredov, Lev and Sonn, Hanson and Petridis, Kalmynin, Kim, Yip, Yoo, Rudnev and Tyrrell, among others; see, for example,
\cite{S12,S13,S14,Sh14,S16,LS17,S20,HP,KYY,Y24,Y25,Kalmynin,KYY26,SY26,RT26}.

In this paper, we develop a unified polynomial framework for studying the following four additive decomposition equations and obtain a complete classification in each case:
\[
 H=A+B,\qquad H\cup\{0\}=A-A,\qquad H=A\rsum A,\qquad H\cup\{0\}=A\rsum A.
\]
Here $A\rsum A=\{a+a':a,a'\in A,\ a\ne a'\}$ denotes the restricted sumset. The first equation is the one appearing in the generalized S\'ark\"ozy conjecture discussed above; the difference-set equation \(H\cup\{0\}=A-A\) was first studied by Lev and Sonn \cite{LS17}, while the two restricted-sumset equations were first studied by Shkredov \cite{Sh14}. Although these four equations are formally similar, they are genuinely distinct problems: none is a formal consequence of another, and each is a natural problem in its own right, with distinct structural features and technical difficulties. Reflecting these differences, the four problems have traditionally been studied separately, using different techniques in each case. 

Hanson and Petridis~\cite{HP} introduced an elegant Stepanov-type auxiliary polynomial (which we call the Hanson--Petridis polynomial) and proved S\'ark\"ozy's conjecture for almost all primes. In the extremal setting, their method yields an exact polynomial factorization that became the starting point for subsequent work. Recently, Kalmynin~\cite{Kalmynin} completed the proof of S\'ark\"ozy's conjecture for quadratic residues. Moreover, he proved the following structural theorem for additive decompositions of arbitrary proper multiplicative subgroups $H$. In particular, the theorem implies that $H$ is additively irreducible unless $|H|$ is a perfect square. 

\begin{theorem}\label{thm:structure}
Let $p$ be an odd prime and let $H<\Fps$ be a proper multiplicative subgroup. Suppose  
\[
H=A+B
\]
for some subsets $A,B\subseteq\Fp$ with $|A|,|B|\geq2$. Set  
\[
 d=|H|,\qquad \alpha=|A|,\qquad \beta=|B|.
\] 
Then $\alpha\beta=d$ and $\alpha=\beta$.
\end{theorem}

Hanson and Petridis~\cite{HP} established the claim $\alpha\beta=d$. To prove $\alpha=\beta$, Kalmynin~\cite{Kalmynin} introduced a new approach based on fractional linear transformations of the Hanson--Petridis factorization, obtaining two reciprocal-sum identities, which he calls Relations~X and~Y. Together with residue calculations, these yield the equal-size theorem and the additional arithmetic information needed to resolve S\'ark\"ozy's conjecture.

Most recently, Rudnev and Tyrrell~\cite{RT26} took Theorem~\ref{thm:structure} as a key structural input. Starting from a modified Hanson--Petridis identity, they derived reciprocal-transfer identities corresponding to Kalmynin's two relations and combined these with symmetric-sum and arithmetic arguments to prove the following classification.\footnote{Before learning of the Rudnev--Tyrrell preprint, the third author independently obtained Theorem~\ref{thm:ordinary} by a residue-theoretic argument derived from Kalmynin's reciprocal relations; see Section~\ref{subsec:comparison}.}

\begin{theorem}\label{thm:ordinary}
Let $p$ be an odd prime and let $H<\Fps$ be a proper multiplicative subgroup. If
\[
H=A+B
\]
for some subsets $A,B\subseteq\Fp$ with $|A|,|B|\geq2$, then $|A|=|B|=2$ and $|H|=4$.
\end{theorem}

Conversely, the case $|H|=4$ is exceptional since $H=\{1,-1,i,-i\}=\{0,-1-i\}+\{1,i\}$, where $i\in H$ has order $4$. The assumption that $H$ is proper is also essential: for example, for every prime $p\ge5$, we have $\Fps=\left\{0,\frac{p-1}{2}\right\}+\left\{1,2,\ldots,\frac{p-1}{2}\right\}$.

We develop a local-to-global differential framework that governs all four decomposition problems considered in this paper. For the sumset equation \(H=A+B\), this framework gives a unified polynomial proof of Theorems~\ref{thm:structure} and~\ref{thm:ordinary} that substantially streamlines these earlier approaches. Conceptually, our local-to-global argument is reminiscent of Stepanov's method \cite{S69, B74}: local vanishing information at many points, combined with degree bounds, forces a global polynomial identity. The resulting global differential identities have different orders, as summarized below.

\begin{center}
\small
\begin{tabular}{c|c|c}
problem & global differential identity & location \\ \hline
$H=A+B$ & second and third order
 & Propositions~\ref{prop:ordinary-second} and~\ref{prop:ordinary-third},
   Equations~\eqref{eq:ordinary-second} and~\eqref{eq:ordinary-third} \\
$H\cup\{0\}=A-A$ & fourth order
 & Proposition~\ref{prop:difference-differential}, Equation~\eqref{eq:difference-fourth} \\
$H=A\rsum A$ & first order
 & Proposition~\ref{prop:restricted-differential}, Equation~\eqref{eq:restricted-first} \\
$H\cup\{0\}=A\rsum A$ & second order
 & Proposition~\ref{prop:restricted-zero-global}, Equation~\eqref{eq:zero-global}
\end{tabular}
\end{center}
All four problems arise from variants of the Hanson--Petridis-type auxiliary-polynomial construction. Lagrange interpolation gives unexpectedly small degree, while the subgroup relation forces high root multiplicity. In the sumset and difference-set problems, these multiplicities account for the full degree and yield exact factorizations; in the restricted-sumset problems, after accounting for the missing diagonal, they instead yield vanishing identities. The first few coefficients beyond the forced multiplicity yield local relations, which degree bounds then lift to the global differential identities summarized above. 

For the sumset equation \(H=A+B\), the resulting second- and third-order differential identities encode, in polynomial form, Kalmynin's Relations~X and~Y. The second-order identity, together with the mixed power-sum relations, gives an independent proof of \(\alpha=\beta\),
while the two identities together reduce the remaining classification
to coefficient comparisons and a short arithmetic argument. In particular, our proof avoids the residue-theoretic input and much of the additional arithmetic analysis of the previous proofs. See Section~\ref{subsec:comparison} for a detailed comparison with the proofs of Kalmynin and Rudnev--Tyrrell. The framework also yields the stability refinement in Theorem~\ref{thm:two-omissions}: if \(H\) is a proper multiplicative subgroup with \(|H|>4\) and \(A+B\subseteq H\) with \(|A|,|B|\ge2\), then \(|A+B|\le |H|-2\), and this bound is sharp in general. 

We next turn to the difference-set equation. Lev and Sonn \cite{LS17} conjectured that, for all sufficiently large primes $p$, one has $\mathcal R_p\cup\{0\}\ne A-A$ for every $A\subseteq\Fp$. Kalmynin \cite{Kalmynin} resolved a more general version of the conjecture, proving the following.

\begin{theorem}\label{thm:difference}
Let $p$ be an odd prime and let $H<\Fps$ be a proper multiplicative subgroup. Suppose
\[
H\cup\{0\}=A-A
\]
for some $A\subseteq\Fp$. Then $|H|\in\{2,6\}$.
\end{theorem}

Kalmynin's proof derives quadratic and cubic reciprocal-sum identities, passes to fractional linear transforms of $A$, and carries out a further polynomial classification. Our local-to-global differential framework gives a considerably shorter proof: the first two local coefficient relations yield a fourth-order differential identity, and a coefficient comparison then completes the classification. The same framework also gives a differential treatment of Kalmynin's ``$d$-critical-pair'' analysis and extends it to the first near-critical case, yielding stability refinements of Theorem~\ref{thm:difference} and an improved upper bound for the clique number of Paley graphs; see Section~\ref{subsec:difference-stability}.

We finally consider the two restricted-sumset equations. The corresponding sumset equations are much easier: if $|H|\ge3$, then neither $H=A+A$ nor $H\cup\{0\}=A+A$ can occur; see \cite[Corollary~1.3]{Y24} and \cite[Proposition~4.1]{Y25}. In the special case $H=\mathcal R_p$, the first assertion also follows from a classical character sum estimate due to Vinogradov; see \cite[Theorem~3.2]{Sh14}. Deleting the diagonal leads naturally to restricted sumsets, classical objects in additive combinatorics, notably through the Erd\H{o}s--Heilbronn problem \cite{DH94,ANR96}. We thus ask when
\[
H=A\rsum A.
\]
For quadratic residues, Shkredov \cite[Theorem~1.2]{Sh14} proved that $\mathcal R_p=A\rsum A$ is impossible for $p>13$. He used quadratic characters to show that a hypothetical $A$ has to be a perfect difference set, and the multiplier theorem together with further character-sum identities then rules out this possibility. However, these ingredients do not seem to extend to general multiplicative subgroups. Using a restricted-sumset variant of the Hanson--Petridis polynomial method, the second author subsequently treated higher-index subgroups, proving that for each fixed index $r\ge3$, the equation $H=A\rsum A$ has no solution for a set of primes $p\equiv1\pmod r$ of lower relative density at least $3/4$, where $H$ is the subgroup of $\Fp^\times$ of index $r$ \cite[Theorem~1.2]{Y25}. Under an additional parity or zero hypothesis, the second author also showed that a hypothetical $A$ is Sidon and $\{2a:a\in A\}\cap H=\varnothing$ \cite[Theorem~1.4]{Y25}. Unlike the sumset problem, deleting the diagonal destroys the exact Hanson--Petridis factorization: the omitted diagonal terms reappear in the local expansions. Tracking these terms, whose form depends on the parity of $|A|$ and, in the zero-augmented variant, on opposite pairs, is the main additional difficulty. The next theorem removes the remaining restrictions and gives a uniform, complete classification over prime fields.

\begin{theorem}\label{thm:restricted}
Let $p$ be an odd prime and let $H<\Fps$ be a proper multiplicative subgroup. Suppose
\[
H=A\rsum A
\]
for some $A\subseteq\Fp$. Then $|H|\in\{1,3,6\}$.
\end{theorem}

A second natural variant is the equation $H\cup\{0\}=A\rsum A$. It is the direct restricted-sumset analogue of the difference-set equation $H\cup\{0\}=A-A$ considered above, and was also studied in \cite{Sh14, Y25}. We characterize it completely.

\begin{theorem}\label{thm:restricted-zero}
Let $p$ be an odd prime and let $H<\Fps$ be a proper multiplicative subgroup. Suppose
\[
H\cup\{0\}=A\rsum A
\]
for some $A\subseteq\Fp$. Then $|H|\in\{2,4\}$.
\end{theorem}

Theorems~\ref{thm:difference},~\ref{thm:restricted}, and~\ref{thm:restricted-zero} are sharp, and Remarks~\ref{rem:difference-exceptions}, \ref{rem:restricted-classification}, and~\ref{rem:restricted-zero-classification} give the full classification in each case.

The Paley graph conjecture provides a broader perspective on these results. In one standard formulation, it predicts that for every $\delta>0$ there exists $\theta=\theta(\delta)>0$ such that, for every sufficiently large prime $p$, all $A,B\subseteq\Fp$ with $|A|,|B|>p^\delta$, and every nontrivial multiplicative character $\chi$ of $\Fp$, satisfy
\[
 \left|\sum_{a\in A,\,b\in B}\chi(a+b)\right| \le p^{-\theta}|A||B|.
\]
This is a central conjecture on character sums; Schoen--Shkredov~\cite{SS22} describe it as ``one of the most important conjectures concerning character sums.'' Standard Fourier-analytic estimates give the conjectured cancellation when $\delta>1/2$, while the range $\delta<1/2$ remains widely open; see \cite{Chang08,SS22} and the references therein. Its depth is illustrated by the fact that it would imply, for example, Vinogradov's conjecture on the least quadratic nonresidue~\cite{Tao15}. By Theorem~\ref{thm:structure}, an exact decomposition $H=A+B$ has $|A|=|B|=\sqrt{|H|}$, so the Paley graph conjecture would already rule out such decompositions whenever $|H|$ exceeds a fixed positive power of $p$. More generally, it predicts that sums of two large sets cannot be anomalously concentrated in a proper multiplicative subgroup. From this perspective, the four decomposition problems considered here may be viewed as extremal instances of the same phenomenon: the relevant sumset, difference set, or restricted sumset is required to lie entirely in \(H\cup\{0\}\) and, in the classification problems, to cover \(H\). Theorems~\ref{thm:ordinary},~\ref{thm:difference},~\ref{thm:restricted}, and~\ref{thm:restricted-zero} settle these four extremal configurations.

\medskip

The paper is organized as follows. Section~\ref{sec:framework} develops the common algebraic and interpolation tools. Section~\ref{sec:sumsets} treats $H=A+B$, derives the second- and third-order differential identities, and proves Theorems~\ref{thm:structure} and~\ref{thm:ordinary}. Section~\ref{sec:difference} proves Theorem~\ref{thm:difference} and develops stability refinements with an application to the clique number of Paley graphs. Finally, Section~\ref{sec:restricted} treats the two restricted-sumset equations and proves Theorems~\ref{thm:restricted} and~\ref{thm:restricted-zero}.

\subsection*{Notation} 
Throughout the paper, we use the following notations. For a set $S\subseteq\Fp$, write
\[
 f_S(X)=\prod_{s\in S}(X-s).
\]
For $j\ge0$, let $p_j(S)$ be the $j$-th power sum of $S$, that is, $p_0(S)=|S|$ and $p_j(S)=\sum_{s\in S}s^j$ for $j\geq 1$. For a polynomial $R \in \Fp[X]$, we write $[X^j]R$ for the coefficient of $X^j$ in $R$. More generally, if a rational function \(R(Y)\in\Fp(Y)\) is regular at \(Y=0\), then \([Y^j]R(Y)\) denotes the coefficient of \(Y^j\) in its formal power-series expansion at \(Y=0\). Unless otherwise stated, polynomial identities are in $\Fp[X]$ and scalar identities involving elements of $\Fp$ are in $\Fp$. When an integer relation is meant modulo $p$, we use congruence notation; when a congruence is lifted to an equality in $\mathbb Z$, we say so explicitly.

\section{A common algebraic framework}\label{sec:framework}
We first record the symmetric-polynomial tools and elementary coefficient expansions used throughout. We then develop the interpolation framework and auxiliary polynomials underlying the four decomposition problems considered in the paper.

\subsection{Symmetric polynomials}
For a set $S\subseteq\Fp$, let $e_j(S)$ denote its $j$th elementary symmetric function, with $e_0(S)=1$. Newton's identities state that
\[
 j e_j(S)=\sum_{i=1}^j(-1)^{i-1}e_{j-i}(S)p_i(S) \qquad(1\le j\le |S|).
\]
\begin{lemma}\label{lem:newton}
Let $S\subseteq\Fp$ have size $m<p$, and let $1\le r\le m$. If $p_j(S)=0$ for $1\le j<r$, then
\[
\begin{aligned}
 [X^{m-j}]f_S&=[X^{m-j}]f_{-S}=0 &&(1\le j<r),\\
 [X^{m-r}]f_S&=-\frac1r p_r(S), &
 [X^{m-r}]f_{-S}&=\frac{(-1)^{r-1}}r p_r(S).
\end{aligned}
\]
In particular, if $|S|\ge2$, then $p_j(S)\ne0$ for some $1\le j\le |S|$.
\end{lemma}

\begin{proof}
Newton's identities give
\[
 e_j(S)=0\quad(1\le j<r), \qquad r e_r(S)=(-1)^{r-1}p_r(S).
\]
Since $[X^{m-j}]f_S=(-1)^je_j(S)$ and $[X^{m-j}]f_{-S}=e_j(S)$, the coefficient formulas follow immediately. Finally, if $p_j(S)=0$ for every $1\le j\le m$, then $f_S=X^m$, which is impossible when $|S|\ge2$.
\end{proof}

If $H<\Fps$ has order $d$, we shall also use the subgroup power-sum identity
\begin{equation}\label{eq:subgroup-power-sum}
 p_j(H)=0\qquad(1\le j<d),
\end{equation}
which follows by multiplying $H$ by an element $h\in H$ with $h^j\ne1$.

\subsection{Coefficient expansions}
We record the elementary coefficient tools used in the local arguments below.

\begin{lemma}\label{lem:coefficient-expansions}
Let $Q\in\Fp[Y]$ satisfy
\[
 Q(Y)\equiv1+q_1Y+q_2Y^2+q_3Y^3\pmod{Y^4},
\]
where $q_1,q_2,q_3\in\Fp$.
For every integer $r\ge0$,
\[
 [Y]Q^r=rq_1, \qquad [Y^2]Q^r=rq_2+\binom r2q_1^2, \qquad [Y^3]Q^r=rq_3+r(r-1)q_1q_2+\binom r3q_1^3.
\]
Moreover, there is a unique polynomial $R\in \Fp[Y]$ of degree at most $3$ such that
\[
 Q(Y)R(Y)\equiv1\pmod{Y^4},
\]
and it is given by
\[
 R(Y)=1-q_1Y+(q_1^2-q_2)Y^2+(-q_1^3+2q_1q_2-q_3)Y^3.
\]
\end{lemma}

\begin{proof}
Put $U=q_1Y+q_2Y^2+q_3Y^3$. We have
\[
 Q^r\equiv(1+U)^r \equiv1+rU+\binom r2U^2+\binom r3U^3 \pmod {Y^4},
\]
\[
 U^2\equiv q_1^2Y^2+2q_1q_2Y^3 \pmod {Y^4}, \qquad U^3\equiv q_1^3Y^3\pmod{Y^4}.
\]
This gives the first three formulas. The formula for $R$ follows by direct multiplication modulo $Y^4$; uniqueness follows by successive coefficient comparison.
\end{proof}

The next lemma records the normalized expansion at a simple root used in the local differential arguments.

\begin{lemma}\label{lem:normalized-expansions}
Assume $p>3$. Let $P\in\Fp[X]$ and let $b\in\Fp$ be a simple root of $P$. Put
\[
 U(Y)=\frac{P(b+Y)}{P'(b)Y}.
\]
For every integer $r\ge0$,
\[
 [Y]U(Y)^r=\frac r2\frac{P''(b)}{P'(b)}, \qquad [Y^2]U(Y)^r=\frac r6\frac{P'''(b)}{P'(b)}+\frac{r(r-1)}8\left(\frac{P''(b)}{P'(b)}\right)^2.
\]
\end{lemma}

\begin{proof}
Since $P(b)=0$, we have
\[
 P(b+Y)\equiv P'(b)Y+\frac{P''(b)}2Y^2+\frac{P'''(b)}6Y^3 \pmod{Y^4},
\]
\[
 U(Y)\equiv 1+\frac{P''(b)}{2P'(b)}Y+\frac{P'''(b)}{6P'(b)}Y^2 \pmod{Y^3}.
\]
Thus Lemma~\ref{lem:coefficient-expansions} gives the required formulas.
\end{proof}

\subsection{Lagrange interpolation and auxiliary polynomials}\label{sec:Lagrange}
For a finite set $S\subseteq\Fp$, put $c_s=1/f_S'(s)$ for $s\in S$.

\begin{lemma}\label{lem:lagrange}
Let $S\subseteq\Fp$ have size $N$. If $R\in\Fp[X]$ has degree at most $N-1$, then
\begin{equation}\label{eq:lagrange}
 \sum_{s\in S}c_sR(s)=[X^{N-1}]R(X).
\end{equation}
\end{lemma}

\begin{proof}
By Lagrange interpolation,
\[
 R(X)=\sum_{s\in S}R(s)\frac{f_S(X)}{(X-s)f_S'(s)}.
\]
Each polynomial $f_S(X)/(X-s)$ is monic of degree $N-1$, so comparison of the coefficients of $X^{N-1}$ gives equation~\eqref{eq:lagrange}.
\end{proof}

The following identity packages the reciprocal sums that occur in the local coefficient arguments. 

\begin{lemma}\label{lem:weighted-partial-fractions}
Let $S\subseteq\Fp$ have size $N$, put $F=f_S$ and $c_a=1/F'(a)$. For $b\in\Fp$ and $0\le j\le N-1$, we have
\begin{equation*}
 \sum_{a\in S}c_a\frac{(b-a)^j}{b+a+Y}=-\frac{(2b+Y)^j}{F(-b-Y)}.
\end{equation*}
\end{lemma}

\begin{proof}
For every polynomial $P\in\Fp[X]$ of degree at most $N-1$, Lagrange interpolation gives
\[
 \frac{P(Z)}{F(Z)}=\sum_{a\in S}c_a\frac{P(a)}{Z-a}.
\]
Apply this identity with $P(X)=(b-X)^j$ and $Z=-b-Y$.
Then $P(-b-Y)=(2b+Y)^j$ and $Z-a=-(b+a+Y)$, so the lemma follows.
\end{proof}

For integers $L,k\ge0$, define
\[
 \Phi_{S;L,k}(X)=\sum_{s\in S}c_s(X+s)^L(X-s)^k.
\]
For $b\in\Fp$, expansion at $X=b+Y$ gives
\begin{equation}\label{eq:phi-local-expansion}
 [Y^j]\Phi_{S;L,k}(b+Y)=\sum_{q+\ell=j}\binom Lq\binom k\ell \sum_{s\in S}c_s(s+b)^{L-q}(b-s)^{k-\ell}.
\end{equation}
We first record the interpolation consequences used below.

\begin{lemma}\label{lem:global-interpolation}
Let $S\subseteq\Fp$ have size $N\ge2$. Then $\deg\Phi_{S;L,k}\le L+k-N+1$ when $L+k\ge N-1$, while $\Phi_{S;L,k}=0$ when $L+k<N-1$.

If in addition $H<\Fps$ has order $d$, $S\subseteq H$, and $u,v\ge0$ satisfy $u+v=N-1$, then
\[
 X^N\mid\Phi_{S;d+u,v}(X)-(-1)^v.
\]
\end{lemma}

\begin{proof}
For each $j$ and each $a\in S$, the coefficient of $X^{L+k-j}$ in $(X+a)^L(X-a)^k$ is a scalar multiple of $a^j$. Hence equation~\eqref{eq:lagrange}, applied to $R(X)=X^j$, shows that the corresponding coefficient of $\Phi_{S;L,k}$ vanishes for $0\le j\le N-2$. This proves the degree assertions.

For the final statement, since $a^d=1$ for $a\in S$, the coefficient of $X^j$ in $\Phi_{S;d+u,v}$, for $0\le j<N$, is
\[
 \left(\sum_{r+\ell=j}(-1)^{v-\ell}\binom{d+u}{r}\binom v\ell \right) \cdot \sum_{a\in S}c_a a^{N-1-j}.
\]
By equation~\eqref{eq:lagrange}, the second factor vanishes for $1\le j<N$ and equals $1$ for $j=0$. Thus the constant term of $\Phi_{S;d+u,v}$ is $(-1)^v$ and the coefficients of $X,\ldots,X^{N-1}$ vanish, proving the divisibility.
\end{proof}

\subsection{Sumset and difference-set factorizations}\label{sec:ordinary-decomposition-identities}
For sumsets and difference sets, interpolation gives an exact global factorization: the forced root multiplicities account for the full degree of the auxiliary polynomial. The following proposition will be used throughout Sections~\ref{sec:sumsets} and~\ref{sec:difference}. Closely related forms appear in Hanson--Petridis~\cite{HP} and Kalmynin~\cite{Kalmynin}.

\begin{proposition}\label{prop:common-factorization}
Let $U,V\subseteq\Fp$ with $\alpha=|U|\ge2$ and $\beta=|V|$, and let $H<\Fps$ have order $d$. Assume $U+V\subseteq H\cup\{0\}$ and $d+\alpha-1<p$, and put $n_0=|U\cap(-V)|$. Writing $c_u=1/f_U'(u)$, define
\[
 \Psi_U(X)=\binom{d+\alpha-1}{\alpha-1}^{-1}\left(\sum_{u\in U}c_u(X+u)^{d+\alpha-1}-1\right).
\]
\begin{enumerate}
\item[(i)] The polynomial $\Psi_U$ is monic of degree $d$. Each $v\in V$ is a root of multiplicity at least $\alpha$, unless $-v\in U$, in which case its multiplicity is exactly $\alpha-1$. In particular, $\alpha\beta-n_0\le d.$

\item[(ii)] If in addition $H\subseteq U+V$, then
\begin{equation}\label{eq:common-size}
 \alpha\beta-n_0=d,
\end{equation}
every element of $H$ has a unique representation $u+v$ with $u\in U$ and $v\in V$, and
\begin{equation}\label{eq:common-factorization}
 \sum_{u\in U}c_u(X+u)^{d+\alpha-1}-1=\binom{d+\alpha-1}{\alpha-1}\prod_{\substack{v\in V\\-v\notin U}}(X-v)^\alpha\prod_{\substack{v\in V\\-v\in U}}(X-v)^{\alpha-1}.
\end{equation}
For every $1\le j<d$,
\begin{equation}\label{eq:common-moments}
 \sum_{i=0}^j\binom ji p_i(U)p_{j-i}(V)=0.
\end{equation}
\end{enumerate}
\end{proposition}

\begin{proof}
By Lemma~\ref{lem:global-interpolation}, $\deg\Psi_U\le d$, while equation~\eqref{eq:lagrange} gives $[X^d]\Psi_U=1$. Thus $\Psi_U$ is monic of degree $d$.

Fix $v\in V$. For $0\le j\le\alpha-2$, put $R_j(X)=(X+v)^{\alpha-1-j}$. For each $u\in U$, we have $(u+v)^{d+\alpha-1-j}=(u+v)^{\alpha-1-j}$ since $\alpha-1-j\ge1$ and $u+v\in H\cup\{0\}$. Applying equation~\eqref{eq:lagrange} to $R_j$ gives
\[
 \sum_{u\in U}c_u(u+v)^{\alpha-1-j}=
 \begin{cases}
  1,&j=0,\\
  0,&1\le j\le\alpha-2.
 \end{cases}
\]
Hence $v$ is a root of $\Psi_U$ of multiplicity at least $\alpha-1$. At order $\alpha-1$, equation~\eqref{eq:lagrange} applied to $R(X)=1$ gives $\sum_{u\in U}c_u=0$. If $-v\notin U$, then the coefficient of $Y^{\alpha-1}$ in $\Psi_U(v+Y)$ is therefore zero, so the multiplicity is at least $\alpha$. If $-v\in U$, the $u=-v$ term in the original coefficient is zero, while every other term contributes $c_u$; hence the coefficient is $\sum_{u\ne -v}c_u=-c_{-v}\ne0$. Thus the multiplicity is exactly $\alpha-1$. Summing these multiplicities over $v\in V$ gives $\alpha\beta-n_0\le d$, proving~(i).

Now assume $H\subseteq U+V$. There are $\alpha\beta-n_0$ ordered pairs $(u,v)\in U\times V$ with $u+v\ne0$, and their sums cover the $d$ elements of $H$. Hence $d\le\alpha\beta-n_0$, while part~(i) gives the reverse inequality. Thus equality holds, and every element of $H$ has a unique representation. The root multiplicities in part~(i) now account for all $d$ roots of $\Psi_U$, so comparison of leading coefficients gives equation~\eqref{eq:common-factorization}. Finally, for $1\le j<d$, uniqueness and equation~\eqref{eq:subgroup-power-sum} give
\[
 \sum_{u\in U,v\in V}(u+v)^j=\sum_{h\in H}h^j=0,
\]
and expanding the left-hand side proves equation~\eqref{eq:common-moments}.
\end{proof}

For restricted sumsets, deleting the diagonal prevents such an exact factorization. Instead, the factor $(X-a)^k$ compensates for the omitted term. For example, for $a,b\in A$, one has $(a+b)^d(b-a)=b-a$ when $A\rsum A \subseteq H$ and $(a+b)^{d+1}(b-a)=(a+b)(b-a)$ when $A\rsum A \subseteq H \cup \{0\}$. We defer the resulting interpolation argument to Section~\ref{sec:restricted}.

\section{Sumsets}\label{sec:sumsets}

We now treat the equation $H=A+B$. We first prove Theorem~\ref{thm:structure}, while deriving the differential identities used later to prove Theorem~\ref{thm:ordinary}. Throughout Sections 3.1-3.4, assume that
\[
H=A+B,\qquad |A|=\alpha\ge2,\qquad |B|=\beta\ge2,\qquad |H|=d.
\]
Since $H$ is proper, $d\le(p-1)/2$. Fixing one element of either summand gives $\alpha,\beta\le d$. Thus every positive integer at most $d+\max\{\alpha,\beta\}-1$ is nonzero in $\Fp$; we use this repeatedly below. In particular, $d+\alpha-1<p$. Put
\[
F=f_B,\qquad G=f_{-A},\qquad c_a=1/f_A'(a).
\]
Since $0\notin H$, we have $A\cap(-B)=\varnothing$ and thus $F$ and $G$ are coprime. Proposition~\ref{prop:common-factorization}(ii), applied with $(U,V)=(A,B)$, gives
\begin{equation}\label{eq:ordinary-size}
 d=\alpha\beta,
\end{equation}
so $d\ge4$ and hence $p>3$. The proposition also shows that every element of $H$ has a unique representation $a+b$, and yields the Hanson--Petridis factorization
\begin{equation}\label{eq:ordinary-factorization}
\sum_{a\in A}c_a(X+a)^{d+\alpha-1}-1=\binom{d+\alpha-1}{\alpha-1}F(X)^\alpha.
\end{equation}

\subsection{Local differential identities}
The exact factorization has a convenient local form which will be the source of both global differential identities.

\begin{lemma}\label{lem:ordinary-coefficient}
For every $b\in B$,
\begin{equation}\label{eq:ordinary-local}
 2(d-1)\frac{G'(b)}{G(b)}=\alpha(\alpha+1)\frac{F''(b)}{F'(b)},
\end{equation}
and
\begin{equation}\label{eq:ordinary-local-second}
 \frac{(d-1)(d-2)}{(\alpha+1)(\alpha+2)}\left(\left(\frac{G'(b)}{G(b)}\right)^2-\frac{G''(b)}{2G(b)}\right)
 =\frac{\alpha}{6}\frac{F'''(b)}{F'(b)}+\frac{\alpha(\alpha-1)}8 \left(\frac{F''(b)}{F'(b)}\right)^2.
\end{equation}
\end{lemma}

\begin{proof}
Put
\[
 U_b(Y)=\frac{F(b+Y)}{F'(b)Y}.
\]
Substituting $X=b+Y$ into equation~\eqref{eq:ordinary-factorization} and comparing the coefficients of $Y^{\alpha+j}$ for $j=0,1,2$ gives
\[
 \binom{d+\alpha-1}{\alpha+j}\sum_{a\in A}\frac{c_a}{(a+b)^{j+1}}=\binom{d+\alpha-1}{\alpha-1}F'(b)^\alpha[Y^j]U_b(Y)^\alpha,
\]
where we used $(a+b)^{d-1-j}=(a+b)^{-j-1}$, since $(a+b)^d=1$ for each $a\in A$. The case $j=0$ is nonzero because $d+\alpha-1<p$ and $F'(b)\ne0$, so we may divide the cases $j=1,2$ by it. 

To evaluate the reciprocal sums, apply Lemma~\ref{lem:weighted-partial-fractions} with $S=A$ and $j=0$. Dividing its value at $Y$ by its value at $Y=0$ gives
\[
 \frac{\displaystyle\sum_{a\in A}\frac{c_a}{a+b+Y}}{\displaystyle\sum_{a\in A}\frac{c_a}{a+b}}=\frac{G(b)}{G(b+Y)}.
\]
Comparing the coefficients of $Y$ and $Y^2$ yields
\[
 \frac{\sum_{a\in A}c_a(a+b)^{-2}}{\sum_{a\in A}c_a(a+b)^{-1}}=\frac{G'(b)}{G(b)}, \qquad
 \frac{\sum_{a\in A}c_a(a+b)^{-3}}{\sum_{a\in A}c_a(a+b)^{-1}}=\left(\frac{G'(b)}{G(b)}\right)^2-\frac{G''(b)}{2G(b)}.
\]
Using also
\[
 \frac{\binom{d+\alpha-1}{\alpha+1}}{\binom{d+\alpha-1}{\alpha}}=\frac{d-1}{\alpha+1}, \qquad
 \frac{\binom{d+\alpha-1}{\alpha+2}}{\binom{d+\alpha-1}{\alpha}}=\frac{(d-1)(d-2)}{(\alpha+1)(\alpha+2)},
\]
the formulas for $[Y]U_b(Y)^\alpha$ and $[Y^2]U_b(Y)^\alpha$ from Lemma~\ref{lem:normalized-expansions} give equations~\eqref{eq:ordinary-local} and~\eqref{eq:ordinary-local-second}.
\end{proof}

The first identity gives the second-order differential relation.

\begin{proposition}\label{prop:ordinary-second}
One has
\begin{equation}\label{eq:ordinary-second}
 \alpha(\alpha+1)GF''+\beta(\beta+1)FG''=2(d-1)F'G'.
\end{equation}
\end{proposition}

\begin{proof}
By symmetry, interchanging $A$ and $B$ and replacing $X$ by $-X$ gives, at every root $z$ of $G$,
\begin{equation}\label{eq:ordinary-local-symmetric}
 2(d-1)\frac{F'(z)}{F(z)}=\beta(\beta+1)\frac{G''(z)}{G'(z)}.
\end{equation}
Let
\[
 \mathcal Q = \alpha(\alpha+1)GF''+\beta(\beta+1)FG''-2(d-1)F'G'.
\]
Equation~\eqref{eq:ordinary-local} shows that $\mathcal Q$ vanishes at every root of $F$, while equation~\eqref{eq:ordinary-local-symmetric} shows that it vanishes at every root of $G$. Since $F$ and $G$ are squarefree and coprime, $FG\mid\mathcal Q$. But $\deg\mathcal Q\le\alpha+\beta-2<\deg(FG)$, so $\mathcal Q=0$.
\end{proof}

\subsection{Equal summand sizes}

We first normalize the summands. By equation~\eqref{eq:common-moments} with $j=1$, we have $\beta p_1(A)+\alpha p_1(B)=0.$ 
Thus translating $(A,B)$ to $(A+\lambda,B-\lambda)$ with $\lambda=-p_1(A)/\alpha$, which preserves $A+B$, we may assume $p_1(A)=p_1(B)=0$. Relabeling the translated sets by $A,B$, all preceding identities remain valid.

\begin{proposition}\label{prop:equal-sizes}
We have $\alpha=\beta$.
\end{proposition}

\begin{proof}
By Lemma~\ref{lem:newton}, there is a least integer $s\ge2$ such that $p_s(A)\ne0$. Put $\tau=p_s(A)$. For $1\le j<s$, the minimality of $s$ reduces equation~\eqref{eq:common-moments} to $\alpha p_j(B)=0,$ so $p_j(B)=0$. Taking $j=s$ similarly gives $\alpha p_s(B)+\beta p_s(A)=0,$ and it follows that $ p_s(B)=-\frac{\beta}{\alpha}\tau.$ Moreover, $s\le\alpha$, and also $s\le\beta$, since otherwise $p_j(B)=0$ for every $1\le j\le\beta$, contrary to Lemma~\ref{lem:newton}. Thus the two summands have the same first nonzero power-sum index $s$ with $2\le s\le\min\{\alpha,\beta\}.$

Applying Lemma~\ref{lem:newton} to $A$ and $B$, we obtain
\[
 F=X^\beta+\xi X^{\beta-s}+\cdots, \qquad G=X^\alpha+\eta X^{\alpha-s}+\cdots,
\]
where
\[
 \xi=\frac{\beta\tau}{\alpha s}, \qquad \eta=\frac{(-1)^{s-1}\tau}{s},
\]
and all coefficients between the displayed terms vanish. Since only the leading term and the displayed lower term of each polynomial can contribute to the coefficient of $X^{\alpha+\beta-s-2}$, we have
\begin{align*}
 [X^{\alpha+\beta-s-2}]GF'' &= (\beta-s)(\beta-s-1)\xi+\beta(\beta-1)\eta,\\
 [X^{\alpha+\beta-s-2}]FG'' &= (\alpha-s)(\alpha-s-1)\eta+\alpha(\alpha-1)\xi,\\
 [X^{\alpha+\beta-s-2}]F'G' &= \alpha(\beta-s)\xi+\beta(\alpha-s)\eta.
\end{align*}
Substituting these expressions into
equation~\eqref{eq:ordinary-second} and using $d=\alpha\beta$, we obtain
\begin{align*}
0={}&
\alpha(\alpha+1)\bigl((\beta-s)(\beta-s-1)\xi+\beta(\beta-1)\eta\bigr)\\
&+\beta(\beta+1)\bigl((\alpha-s)(\alpha-s-1)\eta+\alpha(\alpha-1)\xi\bigr)\\
&-2(\alpha\beta-1)\bigl(\alpha(\beta-s)\xi+\beta(\alpha-s)\eta\bigr).
\end{align*}
Collecting the coefficients of $\xi$ and $\eta$ gives
\[
 0=\alpha s(\alpha s+\alpha-2\beta+s-1)\xi+\beta s(-2\alpha+\beta s+\beta+s-1)\eta.
\]
Using the displayed values of $\xi$ and $\eta$, the preceding identity becomes
\[
 0=\beta\tau\Bigl(\alpha s+\alpha-2\beta+s-1+(-1)^{s-1}(-2\alpha+\beta s+\beta+s-1)\Bigr).
\]
Separating the two parities of $s$ gives
\begin{equation}\label{eq:balance}
 0= \begin{cases}
  \beta(\alpha-\beta)(s+3)\tau,&s\text{ even},\\
  \beta(\alpha+\beta+2)(s-1)\tau,&s\text{ odd}.
 \end{cases}
\end{equation}

Since $d=\alpha\beta\mid p-1$ and $H$ is proper, we can write $p=r\alpha\beta+1$ for some $r\geq 2$. As $2\le s\le\min\{\alpha,\beta\}$ and $\alpha,\beta\ge2$, the integers $s-1$, $s+3$, and $\alpha+\beta+2$ are all positive and less than $p$. Together with $\beta,\tau\ne0$, these bounds show that every factor in equation~\eqref{eq:balance}, except possibly $\alpha-\beta$, is nonzero in $\Fp$. The odd case is therefore impossible, while the even case gives $\alpha-\beta=0$ in $\Fp$. Since $|\alpha-\beta|<p$, we conclude that $\alpha=\beta$.
\end{proof}

Together with equation~\eqref{eq:ordinary-size}, this proves Theorem~\ref{thm:structure}.

\subsection{A third-order differential identity}
We now consider the second local identity \eqref{eq:ordinary-local-second} in Lemma~\ref{lem:ordinary-coefficient} and begin to set up the proof of Theorem~\ref{thm:ordinary}. Put $\alpha = \beta$ for the remainder of the argument.  Then $d=\alpha^2$, and since $H$ is proper we may write $p=r\alpha^2+1$ with $r\ge2$. Equation~\eqref{eq:ordinary-second} becomes
\begin{equation}\label{eq:ordinary-second-equal}
 \alpha(GF''+FG'')=2(\alpha-1)F'G'.
\end{equation}
First, using equation~\eqref{eq:ordinary-local} and the derivative of equation~\eqref{eq:ordinary-second-equal}, we obtain a relation between $GF''' - FG'''$ and $G'F''+F'G''$ at every root of $F$. 
For $b\in B$, put
\[
 u_b=\frac{F''(b)}{F'(b)},\qquad \nu_b=\frac{F'''(b)}{F'(b)},\qquad \rho_b=\frac{G'(b)}{G(b)},\qquad \sigma_b=\frac{G''(b)}{G(b)}.
\]
Equation~\eqref{eq:ordinary-local} gives
\begin{equation}\label{eq:ordinary-local-first-balanced}
 \alpha u_b=2(\alpha-1)\rho_b.
\end{equation}
Differentiating equation~\eqref{eq:ordinary-second-equal}, evaluating at $X=b$, and dividing by $F'(b)G(b)$ gives
\begin{equation}\label{eq:ordinary-local-derivative}
 (2-\alpha)(u_b\rho_b+\sigma_b)+\alpha\nu_b=0.
\end{equation}
For convenience, set
\[
Q_\alpha=3\alpha^2+4\alpha-1, \qquad K_\alpha=\alpha Q_\alpha, \qquad J_\alpha=3(3\alpha^3-2\alpha^2-5\alpha+2).
\]

\begin{proposition}\label{prop:ordinary-third}
One has
\begin{equation}\label{eq:ordinary-third}
K_\alpha(GF'''-FG''')-J_\alpha(G'F''-F'G'')=0.
\end{equation}
\end{proposition}

\begin{proof}
All denominators below are nonzero by the bound at the beginning of the section. Since $d=\alpha^2$, equation~\eqref{eq:ordinary-local-second}
becomes
\begin{equation}\label{eq:ordinary-local-third}
 \frac{(\alpha-1)(\alpha^2-2)}{2(\alpha+2)}(2\rho_b^2-\sigma_b)=\frac{\alpha}{6}\nu_b+\frac{\alpha(\alpha-1)}8u_b^2.
\end{equation}
Equation~\eqref{eq:ordinary-local-first-balanced} gives
\[
 2\rho_b^2=\frac{\alpha}{\alpha-1}u_b\rho_b, \qquad u_b^2=\frac{2(\alpha-1)}{\alpha}u_b\rho_b.
\]
Substituting these into equation~\eqref{eq:ordinary-local-third} and multiplying by $24(\alpha+2)$ gives
\begin{equation}\label{eq:ordinary-local-elimination}
 6(\alpha^3-\alpha-2)u_b\rho_b-12(\alpha-1)(\alpha^2-2)\sigma_b-4\alpha(\alpha+2)\nu_b=0.
\end{equation}
Subtracting equation~\eqref{eq:ordinary-local-elimination} from $3(\alpha^2-3)$ times equation~\eqref{eq:ordinary-local-derivative}
gives
\begin{equation}\label{eq:ordinary-third-local}
 K_\alpha\nu_b-J_\alpha(u_b\rho_b-\sigma_b)=0.
\end{equation}
Let
\[
 \mathcal P = K_\alpha(GF'''-FG''')-J_\alpha(G'F''-F'G'').
\]
Multiplying equation~\eqref{eq:ordinary-third-local} by $F'(b)G(b)$ shows that $\mathcal P$ vanishes at every root of $F$. By symmetry it also vanishes at every root of $G$. Hence $FG\mid\mathcal P$, whereas $\deg\mathcal P\le2\alpha-3<\deg(FG).$ 
Thus $\mathcal P=0$, proving equation~\eqref{eq:ordinary-third}.
\end{proof}

\subsection{Elimination using the differential identities}
We next combine the second- and third-order differential identities through polynomial elimination, obtaining an arithmetic congruence for the common summand size. Our goal here is to prove $Q_\alpha \equiv 0 \pmod{p}$ whenever $\alpha>2$. Together with the condition $\alpha^2\mid p-1$, this will allow us to reduce the $\alpha>2$ case to only $(\alpha,p)=(3,19)$ in the classification.
\begin{proposition}\label{prop:sumset-elimination}
If $\alpha>2$, then
\[
 Q_\alpha=3\alpha^2+4\alpha-1\equiv0\pmod p.
\]
\end{proposition}

\begin{proof}
Put $K=K_\alpha$, $J=J_\alpha$, $M=(\alpha-2)K-\alpha J$, and $N=(\alpha+2)K-\alpha J$. Consider
\[
 D=(\alpha-1)(G'F'''+F'G''')-2(\alpha-2)F''G''.
\]
The coefficients are chosen so that the leading term of $D$ cancels; hence $\deg D\le 2\alpha-5.$ It suffices to prove that
\begin{equation}\label{eq:sumset-D}
 MND=24\alpha^3(\alpha-2)Q_\alpha F''G''.
\end{equation}
Indeed, since $F''G''$ has degree $2\alpha-4$ and $24\alpha^3(\alpha-2)\ne0$, equation~\eqref{eq:sumset-D} implies that $Q_\alpha=0$.

It remains to prove equation~\eqref{eq:sumset-D}. We work modulo $F$.
Equation~\eqref{eq:ordinary-second-equal} gives
\begin{equation}\label{eq:sumset-D1}
 \alpha GF''\equiv2(\alpha-1)F'G'\pmod F.
\end{equation}
Differentiating equation~\eqref{eq:ordinary-second-equal} once and twice, and reducing modulo $F$, gives
\begin{align}
 \alpha GF''' &\equiv(\alpha-2)(G'F''+F'G'') \pmod F, \label{eq:sumset-second-first-derivative}\\
 \alpha GF'''' &\equiv-2(G'F'''+F'G''')+2(\alpha-2)F''G'' \pmod F. \label{eq:sumset-second-second-derivative}
\end{align}
Similarly, equation~\eqref{eq:ordinary-third} and its derivative give
\begin{align}
 KGF''' &\equiv J(G'F''-F'G'') \pmod F, \label{eq:sumset-third-mod}\\
 KGF'''' &\equiv-(K-J)(G'F'''-F'G''') \pmod F. \label{eq:sumset-third-derivative-mod}
\end{align}
Subtracting $(\alpha-2)$ times equation~\eqref{eq:sumset-third-mod} from $J$ times equation~\eqref{eq:sumset-second-first-derivative} gives
\begin{equation}\label{eq:sumset-D2}
 MGF'''\equiv-2(\alpha-2)JF'G''\pmod F.
\end{equation}
Likewise, subtracting $K$ times equation~\eqref{eq:sumset-second-second-derivative} from $\alpha$ times equation~\eqref{eq:sumset-third-derivative-mod} gives
\begin{equation}\label{eq:sumset-D3}
 NF'G''' \equiv MG'F'''+2(\alpha-2)KF''G''\pmod F.
\end{equation}

We now compute $MNGD$ modulo $F$. Using equation~\eqref{eq:sumset-D3} and the definition of $D$ first gives
\[
 MNGD \equiv (\alpha-1)M(M+N)GG'F'''+2(\alpha-2) \bigl((\alpha-1)KM-MN\bigr)GF''G'' \pmod F.
\]
Next use equation~\eqref{eq:sumset-D2} to replace $MGF'''$, and then equation~\eqref{eq:sumset-D1} to replace $F'G'$.
Since $M+N=2\alpha(K-J)$, the above equation becomes
\begin{equation}\label{eq:sumset-D-reduction}
 MNGD \equiv 2(\alpha-2) \bigl((\alpha-1)KM-MN-\alpha^2J(K-J)\bigr) GF''G'' \pmod F.
\end{equation}
Observe that $(\alpha-1)KM-MN-\alpha^2J(K-J) = K\bigl(\alpha J-3(\alpha-2)K\bigr)=12\alpha^3Q_\alpha,$ where we used $K=\alpha Q_\alpha$ and $J-3(\alpha-2)Q_\alpha=12\alpha$.
Thus equation~\eqref{eq:sumset-D-reduction} simplifies to \begin{equation}\label{eq:sumset-D-mod-F}
 MNGD \equiv 24\alpha^3(\alpha-2)Q_\alpha\,GF''G'' \pmod F.
\end{equation}
Since $F$ and $G$ are coprime, we may cancel $G$ modulo $F$, proving equation~\eqref{eq:sumset-D} modulo $F$.

By symmetry, equation~\eqref{eq:sumset-D} also holds modulo $G$. Thus the difference between the two sides of equation~\eqref{eq:sumset-D} is divisible by $FG$, while its degree is at most $2\alpha-4<\deg(FG)$.
Therefore equation~\eqref{eq:sumset-D} holds identically, completing the proof.
\end{proof}

We can now complete the classification.

\begin{proof}[Proof of Theorem~\ref{thm:ordinary}]
By Proposition~\ref{prop:equal-sizes}, $d=\alpha^2$. Suppose for a contradiction that $\alpha>2$. Proposition~\ref{prop:sumset-elimination} gives
\[
 Q_\alpha=3\alpha^2+4\alpha-1\equiv0\pmod p.
\]
If $\alpha\ge4$, then $p\ge2\alpha^2+1$ and
\[
 0<Q_\alpha<2p, \qquad 2p-Q_\alpha\ge(\alpha-1)(\alpha-3)>0.
\]
Thus $p=Q_\alpha=3\alpha^2+4\alpha-1$. Since $\alpha^2=d\mid p-1$, we have $\alpha^2\mid4\alpha-2$, hence $\alpha\mid2$, a contradiction. Therefore $\alpha=3$.

Now $Q_3=38$, and we have $p=19$. Since $p_1(A)=p_1(B)=0$, for some $u,v,w,z\in\Fp$,
\[
 F=X^3+uX+v,\qquad G=X^3+wX+z.
\]
In characteristic $19$, equation~\eqref{eq:ordinary-second-equal} is $3(GF''+FG'')=4F'G'$, and substitution gives
\[
 6(u+w)X^2+18(v+z)X-4uw=0.
\]
Hence $w=-u$, $z=-v$, and $uw=0$, so $u=w=0$ and therefore $F=X^3+v$, $G=X^3-v$. If $a\in A$, then $G(-a)=0$, so $F(a)=0$.
Thus $A\subseteq B$, and hence $A=B$. For distinct $a,a'\in A$, the sum $a+a'\in H$ then has the two representations $(a,a')$ and $(a',a)$, contradicting Proposition~\ref{prop:common-factorization}(ii).
Therefore $\alpha=2$ and $d=4$.
\end{proof} 

\subsection{A stability refinement}
Theorem~\ref{thm:ordinary} rules out a nontrivial sumset equal to $H$ when $|H|>4$. The same framework also rules out a sumset missing a single element of $H$.

\begin{theorem}\label{thm:two-omissions}
Let $p$ be prime, and let $H$ be a proper multiplicative subgroup of $\mathbb F_p^\times$ with $|H|>4$.
If $A,B\subseteq\mathbb F_p$ satisfy $|A|,|B|\ge2$ and $A+B\subseteq H$, then
\[
 |A+B|\le |H|-2.
\]
\end{theorem}

\begin{proposition}\label{prop:one-element-deletion}
Let $p$ be prime, let $H$ be a proper multiplicative subgroup of $\mathbb F_p^\times$, and let $h\in H$.
Then $H\setminus\{h\}$ is additively irreducible: there are no $A,B\subseteq\mathbb F_p$ with $|A|,|B|\ge2$ such that
\[
 H\setminus\{h\}=A+B.
\]
\end{proposition}

\begin{proof}
Suppose that such a decomposition exists, and put $d=|H|$, $\alpha=|A|$, $\beta=|B|$, and $q=\alpha\beta$.
For $S=A$ or $B$, with $m=|S|$, use the auxiliary polynomial of Section~\ref{sec:Lagrange} to define
\[
 \Psi_S(X)=\binom{d+m-1}{m-1}^{-1} \bigl(\Phi_{S;d+m-1,0}(X)-1\bigr).
\]
Fixing an element of either summand gives $\alpha,\beta\le d$. Since $2d\le p-1$, we have $d+\max\{\alpha,\beta\}-1<p$, so the binomial coefficients above are nonzero. Proposition~\ref{prop:common-factorization}(i), applied to $(A,B)$, shows that every $b\in B$ is a root of $\Psi_A$ of multiplicity at least $\alpha$. Hence
\[
 \Psi_A=f_B^\alpha R,\qquad \deg R=d-q,
\]
for a monic polynomial $R$. Exchanging $A$ and $B$ gives a second factorization. In particular, $q\le d$. Since $d-1=|A+B|\leq q$, we have $q\in \{d-1,d\}$. Moreover, $d\ge4$ since $\alpha,\beta\ge2$.

To treat these two cases, we examine the coefficients of $X^{d-1}$ and $X^{d-2}$ in $\Psi_S$. Put $c_s=1/f'_S(s)$, as in Section~\ref{sec:Lagrange}, and apply equation~\eqref{eq:lagrange} to each of the following polynomials:
\[
 X^{m-1},\qquad X^m-f_S(X),\qquad X^{m+1}-(X+p_1(S))f_S(X).
\]
Note that each has degree at most $m-1$. Using $2e_2(S)=p_1(S)^2-p_2(S)$, we obtain
\[
 \sum_{s\in S}c_s s^{m-1}=1,\qquad \sum_{s\in S}c_s s^m=p_1(S),\qquad \sum_{s\in S}c_s s^{m+1}=\frac{p_1(S)^2+p_2(S)}2.
\]
Substitution in the definition of $\Psi_S$ gives
\begin{equation}\label{eq:deletion-top-coefficients}
 \Psi_S(X)=X^d+\frac d m p_1(S)X^{d-1}+\frac{d(d-1)}{2m(m+1)}\bigl(p_1(S)^2+p_2(S)\bigr)X^{d-2}+\cdots.
\end{equation}
Write
\[
 \kappa_A=\frac{p_1(A)}\alpha,\qquad \kappa_B=\frac{p_1(B)}\beta,\qquad \kappa=\kappa_A+\kappa_B.
\]

If $q=d$, then $R=1$, and comparing the coefficients of $X^{d-1}$ in $\Psi_A=f_B^\alpha$ gives $d\kappa_A=-q\kappa_B$, hence $\kappa=0$. Also, since $q = d$, the equality $|A+B|=d-1$ implies exactly one element $z\in A+B$ has two representations; every other element has one. Because $\sum_{x\in H}x=0$ and $h\notin A+B$, this gives the contradiction
\[
 0=q\kappa=\sum_{a\in A}\sum_{b\in B}(a+b)=z-h\ne0.
\]

It remains to consider $q=d-1$. Put
\[
 v_A=\frac{p_2(A)}\alpha-\kappa_A^2,\qquad
 v_B=\frac{p_2(B)}\beta-\kappa_B^2.
\]
Every element of $H\setminus\{h\}$ now has a unique representation as $a+b$. Thus, for $j=1,2$,
\[
 \sum_{a\in A}\sum_{b\in B}(a+b)^j  =\sum_{x\in H}x^j-h^j  =-h^j.
\]
Expanding the left-hand side for $j=1$ and $j=2$ gives, respectively,
\[
 q\kappa=-h,\qquad q(v_A+v_B+\kappa^2)=-h^2.
\]
Consequently $\kappa\ne0$. Since $q=d-1$ and $h=-q\kappa$, the second identity gives $v_A+v_B=-(q+1)\kappa^2=-d\kappa^2$. 

Comparing the coefficients of $X^{d-1}$ in $\Psi_A=f_B^\alpha R$ gives $R=X+d\kappa_A+q\kappa_B$. We also have
\[
 f_B^\alpha=X^q-q\kappa_BX^{q-1}+\frac q2\bigl((q-1)\kappa_B^2-v_B\bigr)X^{q-2}+\cdots.
\]
Then, comparing the coefficients of $X^{d-2}$ in $\Psi_A=f_B^\alpha R$ and using equation~\eqref{eq:deletion-top-coefficients} gives the first identity below. Exchanging $A$ and $B$ gives the second:
\[
 \begin{aligned}
 d v_A+(\alpha+1)v_B&=-d(\alpha+1)\kappa^2,\\
 (\beta+1)v_A+d v_B&=-d(\beta+1)\kappa^2.
 \end{aligned}
\]
Subtracting $(\alpha+1)$ and $(\beta+1)$ times $v_A+v_B=-d\kappa^2$ from the respective equations yields
\[
 \alpha(\beta-1)v_A=0,\qquad \beta(\alpha-1)v_B=0.
\]
Since $2\le\alpha,\beta<p$, these equations give $v_A=v_B=0$, contradicting $d\kappa^2\ne0$.
\end{proof}

\begin{proof}[Proof of Theorem~\ref{thm:two-omissions}]
Theorem~\ref{thm:ordinary} gives $A+B\ne H$, so $|A+B|\le |H|-1$.
Equality would give $A+B=H\setminus\{h\}$ for some $h\in H$, contrary to Proposition~\ref{prop:one-element-deletion}.
\end{proof}

\begin{remark}\label{rem:sumset-product-sharp}
The bound in Theorem~\ref{thm:two-omissions} cannot be improved uniformly. Indeed, for the quadratic residues $\mathcal{R}_{13}$ in $\mathbb F_{13}$, we have $\mathcal{R}_{13}\setminus\{1,12\}=\{0,1\}+\{3,9\}.$

Also, the conclusion of Theorem~\ref{thm:two-omissions} cannot in general be strengthened by replacing $|A+B|$ with $|A||B|$. Indeed, let $p\ge13$ satisfy $p\equiv1\pmod4$, let $H=\mathcal R_p$, and choose $x\in \Fps 
\setminus \mathcal R_p$. Put $A=\{0,x\}$ and $B=H\cap(H-x)$; then $A+B\subseteq H$. If $\chi$ denotes the quadratic character of $\Fp$, then $x$ and $-x$ are both nonresidues, and hence
\[
 |B|=\frac14\sum_{t\in\Fp}(1+\chi(t))(1+\chi(t+x))
 =\frac{p-1}{4}=\frac{|H|}{2},
\]
where we used the standard identity $\sum_{t\in\Fp}\chi(t(t+x))=-1$. Consequently $|A||B|=|H|$. Thus the bound $|A||B|\le |H|$ in Proposition~\ref{prop:common-factorization}(i) is sharp under the assumption $A+B\subseteq H$.
\end{remark}

\subsection{Comparison with previous approaches}\label{subsec:comparison}
It is useful to compare the preceding proof with the arguments of Kalmynin and of Rudnev--Tyrrell. Kalmynin first proves the equality $\alpha=\beta$ for additive decompositions of arbitrary proper multiplicative subgroups~\cite[Theorem~2]{Kalmynin}. The key local input is the pair of reciprocal-sum identities, Relations~X and~Y, from \cite[Lemma~10]{Kalmynin}. Relation~X, together with the residue theorem applied to suitable rational differential forms on $\mathbb P^1$, is the main ingredient in the proof of the equal-size theorem. In \cite[Section~5]{Kalmynin}, Kalmynin specializes to the quadratic-residue subgroup $\mathcal R_p$ and uses Relation~Y, together with further residue calculations, to derive the additional arithmetic information needed to prove S\'ark\"ozy's conjecture. In this specialization,
\[
 d=\alpha^2=\frac{p-1}{2}, \qquad \alpha^2\equiv-\frac12\pmod p.
\]

Temporarily following Kalmynin's notation, let
\[
 n:=\min\{j\geq1:p_j(A)\neq0\}.
\]
If there exists an integer $j\leq\alpha$ such that $p_j(A)\neq0$ and $n\nmid j$, define
\[
 m:=\min\{1\leq j\leq\alpha:p_j(A)\neq0,\ n\nmid j\}.
\]
For $k=n$, and also for $k=m$ when $m$ exists, \cite[Lemma~13]{Kalmynin} gives, after specialization to quadratic residues, the congruence
\[
 2(3k-2)(k-1)\alpha+(k+2)(k+3)\equiv0\pmod p.
\]

A closer inspection of the proof of \cite[Lemma~13]{Kalmynin}, together with \cite[Lemmas~14--16]{Kalmynin}, shows that the quadratic-residue condition is not used in the preceding residue calculations. For the remainder of this comparison, assume $\alpha > 2$. Keeping the coefficients in their unspecialized form and substituting only $d=\alpha^2$, the congruence on \cite[p.~29]{Kalmynin} becomes
\begin{equation}\label{eq:kalmynin-residue}
 \left(1-\frac{\gamma_1(\alpha-1)}{\gamma_0^2}\right) \left(\frac{2}{\gamma_0}-1\right)^{-1} (2\gamma_5+\gamma_4) \equiv 2\gamma_1\gamma_2-\gamma_4
 \pmod p,
\end{equation}
where
\[
 \gamma_0=\frac{\alpha}{\alpha-1}, \qquad \gamma_1=\frac{\alpha(\alpha+2)}{(\alpha-1)(\alpha^2-2)}, \qquad \gamma_2=\alpha^2-(k+2)\alpha+\frac{(k+1)(k+2)}{3},
\]
\[
 \gamma_3=\alpha-\frac{k+1}{2}, \qquad \gamma_4=(k+2)\gamma_0\gamma_3-k\alpha, \qquad \gamma_5=\alpha^2-(k+2)\gamma_0\gamma_3.
\]
The congruence $\alpha^2\equiv-\frac12\pmod p$ is used only afterward, when equation~\eqref{eq:kalmynin-residue} is specialized to the quadratic-residue case. In fact, equation~\eqref{eq:kalmynin-residue} is equivalent to
\[
 \frac{\alpha}{6(\alpha-1)(\alpha^2-2)}\Lambda_\alpha(k)\equiv0\pmod p,
\]
where
\[
 \Lambda_\alpha(X):=Q_\alpha(X-1)(X-2)-12\alpha(\alpha-X), \qquad Q_\alpha=3\alpha^2+4\alpha-1.
\]
Since $\alpha^2=d\mid p-1$, $H$ is proper, and $\alpha>2$, we have $p\ge2\alpha^2+1$. It follows that
\begin{equation}\label{eq:kalmynin-congruence}
 \Lambda_\alpha(k)\equiv0\pmod p
\end{equation}
for $k=n$, and also for $k=m$ when $m$ exists. This observation, together with additional arithmetic analysis, gives the residue-theoretic route by which the third author independently obtained Theorem~\ref{thm:ordinary} before learning of the preprint of Rudnev and Tyrrell.

Rudnev and Tyrrell \cite[Proposition~5.1]{RT26} obtain the same equation~\eqref{eq:kalmynin-congruence} by a different argument. They use Kalmynin’s equal-size theorem together with a modified Hanson--Petridis polynomial identity to derive first- and second-order reciprocal-transfer identities corresponding to Kalmynin’s Relations X and Y. Symmetric-sum calculations and a root-uniqueness argument then establish the required power-sum support, followed by a separate arithmetic analysis of the congruence.

The two identities in Lemma~\ref{lem:ordinary-coefficient} are polynomial forms of Kalmynin's Relations~X and~Y~\cite[Lemma~10]{Kalmynin}, so the underlying local reciprocal information is already present in his argument. Kalmynin exploits this information through fractional linear transformations of the Hanson--Petridis factorization and residue calculations, whereas we lift it directly to the global second- and third-order differential identities \eqref{eq:ordinary-second} and \eqref{eq:ordinary-third}. Polynomial elimination then gives \(Q_\alpha\equiv0\pmod p\), from which the classification follows by a short arithmetic argument. Thus the proof avoids Kalmynin's residue-theoretic stage, as well as the reciprocal-transfer, power-sum support, root-uniqueness, and separate arithmetic stages in the argument of Rudnev and Tyrrell, giving a substantially shorter route to the classification.

\section{Difference sets}\label{sec:difference}
\subsection{The equation \texorpdfstring{$H\cup\{0\}=A-A$}{H∪{0}=A-A}}
Assume
\[
H\cup\{0\}=A-A, \qquad n=|A|, \qquad d=|H|.
\]
Put $F=f_A$ and $c_a=1/F'(a)$. Fixing one element of $A$ gives $n-1\le d$, and properness of $H$ gives $2d\le p-1$; hence $d+n-1<p$. Applying Proposition~\ref{prop:common-factorization}(ii) with $(U,V)=(-A,A)$ gives
\begin{equation}\label{eq:difference-size}
 d=n(n-1),
\end{equation}
and every nonzero difference has a unique ordered representation. 

Since $f_{-A}'(-a)=(-1)^{n+1}F'(a)$ for each $a\in A$, equation~\eqref{eq:common-factorization} becomes
\begin{equation}\label{eq:difference-factorization}
\sum_{a\in A}c_a(X-a)^{n^2-1}-(-1)^{n-1}=(-1)^{n-1}\binom{n^2-1}{n-1}F(X)^{n-1}.
\end{equation}
Since $H$ is proper,
\begin{equation}\label{eq:n-square-bound}
p\ge2d+1=2n(n-1)+1>n^2.
\end{equation}

For each \(b\in A\), substitute \(X=b+Y\) in the factorization identity. After normalizing by the nonzero \(Y^{n-1}\)-coefficient, the coefficients of \(Y^{n+1}\) and \(Y^{n+2}\) yield two local derivative relations, which combine to give the following fourth-order differential identity.

\begin{proposition}\label{prop:difference-differential}
Assume $n\ge4$. Then
\begin{equation}\label{eq:difference-fourth}
4n(n-2)F'F'''-3(n-1)(n-2)(F'')^2-n(n+1)FF''''=0.
\end{equation}
\end{proposition}

\begin{proof}
Fix $b\in A$ and write
\[
 P_b(Y)=\frac{F(b+Y)}{F'(b)Y} \equiv 1+q_1Y+q_2Y^2+q_3Y^3\pmod{Y^4}.
\]
Since $P_b(0)=1$, the function $1/P_b(Y)$ is regular at $Y=0$.
The partial-fraction decomposition of $1/F$ gives
\[
 \frac1{F(b+Y)}=\frac{c_b}{Y}+\sum_{a\in A\setminus\{b\}}\frac{c_a}{b-a+Y},
\]
and it follows from $F(b+Y)=F'(b)YP_b(Y)$ that 
\begin{equation}\label{eq:difference-local-pf}
 \sum_{a\in A\setminus\{b\}}\frac{c_a}{b-a+Y}=\frac{c_b}{Y}\left(\frac1{P_b(Y)}-1\right).
\end{equation}

Substitute $X=b+Y$ into equation~\eqref{eq:difference-factorization}. We first compare the coefficients of $Y^{n-1}$. On the left-hand side, since $(b-a)^d=1$ for $a\ne b$ and $\sum_{a\in A}c_a=0$, the coefficient is
\[
 \binom{n^2-1}{n-1}\sum_{a\ne b}c_a=-\binom{n^2-1}{n-1}c_b.
\]
On the other hand,
\[
 F(b+Y)^{n-1}=F'(b)^{n-1}Y^{n-1}P_b(Y)^{n-1},
\]
with $P_b(0)=1$. Hence
\begin{equation}\label{eq:difference-base-coeff}
 -c_b=(-1)^{n-1}F'(b)^{n-1}.
\end{equation}

Now let $1\le\ell\le3$. Since $(b-a)^d=1$ for $a\ne b$, the $Y^{n-1+\ell}$-coefficient on the left-hand side of equation~\eqref{eq:difference-factorization} is
\[
 \binom{n^2-1}{n-1+\ell}\sum_{a\in A\setminus\{b\}}c_a(b-a)^{d-\ell}=\binom{n^2-1}{n-1+\ell}\sum_{a\in A\setminus\{b\}}c_a(b-a)^{-\ell}.
\]
Taking the coefficient of $Y^{\ell-1}$ in equation~\eqref{eq:difference-local-pf} gives 
\[
 \sum_{a\in A\setminus\{b\}}c_a(b-a)^{-\ell}=(-1)^{\ell-1}c_b[Y^\ell]\frac1{P_b(Y)}.
\]
Comparing with the right-hand side of equation~\eqref{eq:difference-factorization} and using
equation~\eqref{eq:difference-base-coeff} therefore gives 
\begin{equation}\label{eq:difference-local-coeff}
 (-1)^\ell \prod_{j=0}^{\ell-1}\frac{d-j}{n+j} [Y^\ell]\frac1{P_b(Y)} = [Y^\ell]P_b(Y)^{n-1},
\end{equation}
where we used
\[
 \frac{\binom{n^2-1}{n-1+\ell}}{\binom{n^2-1}{n-1}}=\prod_{j=0}^{\ell-1}\frac{d-j}{n+j}.
\]
For $\ell=1$, equation~\eqref{eq:difference-local-coeff} is automatically satisfied since $d=n(n-1)$, so the first new relations arise from $\ell=2,3$.

Applying Lemma~\ref{lem:coefficient-expansions} to equation~\eqref{eq:difference-local-coeff} and using $d=n(n-1)$, the case $\ell=2$ gives
\[
 \frac{n(n-1)}{2(n+1)}\bigl((n-1)q_1^2-2nq_2\bigr)=0,
\]
and hence
\begin{equation}\label{eq:difference-q2}
 2nq_2=(n-1)q_1^2.
\end{equation}
For $\ell=3$, after using equation~\eqref{eq:difference-q2}, we obtain
\[
 \frac{(n-1)(n-3)}{6(n+2)}\bigl(6n^2q_3-(n-1)(n-2)q_1^3\bigr)=0,
\]
and therefore
\begin{equation}\label{eq:difference-q3}
 6n^2q_3=(n-1)(n-2)q_1^3.
\end{equation}
All canceled factors are nonzero: indeed $n\ge4$, while inequality~\eqref{eq:n-square-bound} gives $p>n^2$.

Since
\[
 q_1=\frac{F''(b)}{2F'(b)},\qquad q_2=\frac{F'''(b)}{6F'(b)},\qquad q_3=\frac{F''''(b)}{24F'(b)},
\]
equation~\eqref{eq:difference-q2} becomes
\begin{equation}\label{eq:difference-local-third}
 4nF'(b)F'''(b)=3(n-1)F''(b)^2,
\end{equation}
and combining equations~\eqref{eq:difference-q2} and \eqref{eq:difference-q3} gives
\begin{equation}\label{eq:difference-local-fourth}
 3nF'(b)F''''(b)=2(n-2)F''(b)F'''(b).
\end{equation}

Let $E$ denote the left-hand side of equation~\eqref{eq:difference-fourth}. Since $F(b)=0$, equation~\eqref{eq:difference-local-third} gives $E(b)=0$. Differentiating $E$ and again using $F(b)=0$, we obtain
\[
 E'(b)=(n-3)\bigl(3nF'(b)F''''(b)-2(n-2)F''(b)F'''(b)\bigr)=0
\]
by equation~\eqref{eq:difference-local-fourth}. Thus every root of $F$ is a double root of $E$. Since $F$ is squarefree, $F^2\mid E$. Finally, $\deg E\le 2n-4<2n=\deg F^2$, and we conclude that $E=0$, as required.
\end{proof}

\begin{proof}[Proof of Theorem~\ref{thm:difference}]
The cases $n=2,3$ give $d=2,6$ by equation~\eqref{eq:difference-size}. Suppose $n\ge4$. Since $n<p$ by inequality~\eqref{eq:n-square-bound}, translate $A$ so that $p_1(A)=0$. Let $s=\min\{j\ge2:p_j(A)\ne0\}$ and $\tau=p_s(A)$. Equation~\eqref{eq:common-moments}, applied to $(-A,A)$, gives $n(1+(-1)^s)\tau=0$, so $s$ is odd. By Lemma~\ref{lem:newton},
\[
 F(X)=X^n+\xi X^{n-s}+\cdots, \qquad \xi=-\frac{\tau}{s}\ne0,
\]
with all intervening coefficients zero.

Comparing the coefficient of $X^{2n-s-4}$ in equation~\eqref{eq:difference-fourth} and collecting powers of $s$ gives
\[
\begin{aligned}
0&=n\xi\bigl(-(n+1)s^4+6(n-1)s^3+(n+1)s^2-6(n-1)s\bigr)\\
&=-ns(s^2-1)\bigl((n+1)s-6(n-1)\bigr)\xi=-ns(s-1)(s+1)\bigl((s-6)n+s+6\bigr)\xi.
\end{aligned}
\]

Since \(3\le s\le n\), \(\xi\ne0\), and inequality~\eqref{eq:n-square-bound} gives \(p>n^2\), all factors outside the last parenthesis are nonzero. Thus, $(s-6)n+s+6=0$. Its absolute value is less than $n^2<p$, hence $(6-s)(n+1)=12$ in $\mathbb Z$. Since $s\ge3$ is odd and $n\ge4$, necessarily $(n,s)=(11,5)$.

It remains to exclude this case. Now $d=110$, $p\ge221$, and
\[
 F(X)=X^{11}+\xi X^6+G(X),\qquad \deg G\le5,\quad \xi\ne0.
\]
It follows that
\[
 X^{15}=X^4X^{11}\equiv-\xi X^{10}-X^4G(X) \pmod F.
\]
Thus the remainder of $X^{15}$ modulo $F$ has degree at most $10$ and leading coefficient $-\xi$. Since this remainder agrees with $X^{15}$ on $A$, equation~\eqref{eq:lagrange} gives $\sum_{a\in A}c_a a^{15}=-\xi.$ Comparing the coefficient of $X^{105}$ in equation~\eqref{eq:difference-factorization}: the left-hand side contributes $-\binom{120}{15}\sum_{a\in A}c_a a^{15}=\binom{120}{15}\xi$; on the right-hand side, the coefficient is $10\binom{120}{10}\xi$, obtained by choosing $\xi X^6$ from one of the ten factors of $F(X)^{10}$ and $X^{11}$ from the other nine. Hence $\binom{120}{15}\xi=10\binom{120}{10}\xi$. Since $p\ge221>120$, both binomial coefficients are nonzero in $\Fp$, and $\xi\ne0$; thus
\[
 10=\frac{\binom{120}{15}}{\binom{120}{10}}=\frac{3708834}{91}\qquad\text{in }\Fp.
\]
Thus $p\mid3707924=2^2\cdot11^2\cdot47\cdot163,$ which is impossible for $p\ge221$. Therefore $n\in\{2,3\}$, and equation~\eqref{eq:difference-size} gives $d\in\{2,6\}$.
\end{proof}

\begin{remark}\label{rem:difference-exceptions}
The proof gives the full classification. If $d=2$, then $A=x+\{0,1\}$ for some $x\in\Fp$. If $d=6$, let $\omega\in H$ be a primitive sixth root of unity. Since $n=3$, after translating $A$ and multiplying by an element of $H$ we may write $A=\{0,1,u\}$. Uniqueness of the nonzero differences gives $\{\pm u,\pm(u-1)\}=\{\pm\omega,\pm\omega^2\}.$
As $\omega^2=\omega-1$, this forces $u=\omega$ or $\omega^{-1}$. Thus, after replacing $\omega$ by $\omega^{-1}$ if necessary, $ A=x+h\{0,1,\omega\}$ for some $x\in\Fp$ and $h\in H$. Conversely, every such set satisfies $A-A=H\cup\{0\}$.
\end{remark}

\subsection{Stability and clique number of Paley graphs}\label{subsec:difference-stability}

Our framework also controls the near-extremal cases of the general bound $|A|(|A|-1)\le |H|$ from Proposition~\ref{prop:common-factorization}(i). This yields a stability refinement of Theorem~\ref{thm:difference}, together with a stronger bound for $|A|(|A|-1)$ when $H$ is sufficiently large relative to its index.

\begin{theorem}\label{thm:difference-stability}
Let $H<\Fps$ be proper of order $d>6$, let $A\subseteq\Fp$ satisfy $A-A\subseteq H\cup\{0\}$, and put $n=|A|$ and $r=(p-1)/d$. Then
\begin{enumerate}
\item[(i)] $|A-A|\le d-3$.
\item[(ii)] Unless $(p,d)=(41,20)$, one has $n(n-1)\le d-2$.
\item[(iii)] If $d>16r^2+4r+2$, then $n(n-1)\le d-4$. In particular, this holds if $d>3p^{2/3}$.
\end{enumerate}
\end{theorem}

The bound in~(i) is sharp. Indeed, if $H=\mathcal R_{17}$ and $A=\{0,1,2\}$, then $A-A=\{0,\pm1,\pm2\}\subseteq H\cup\{0\}$ and $|A-A|=5=|H|-3$.

\begin{lemma}\label{lem:difference-near-critical}
Let $H<\Fps$ be proper of order $d$ and index $r=(p-1)/d$, and let $A\subseteq\Fp$ have size $n\ge2$ with $A-A\subseteq H\cup\{0\}$. Put $q=n(n-1)$.
\begin{enumerate}
\item[(i)] If $q=d$, then $d\in\{2,6\}$ or $(p,d)=(41,20)$.
\item[(ii)] If $q=d-2$, then, after translating $A$, one has $A=-A$, either $n+1\mid p-1$ or $n+1\mid p+1$, and $d\le16r^2+4r+2$.
\end{enumerate}
\end{lemma}

\begin{proof}
Fixing one element of $A$ gives $n-1\le d$, while properness of $H$ gives $2d\le p-1$; hence $d+n-1<p$. Translate $A$ so that $p_1(A)=0$, and put $F=f_A$.

Suppose first that $q=d$. Proposition~\ref{prop:common-factorization}(i), applied to $(-A,A)$, shows that the forced root multiplicities account for the full degree, so the exact factorization~\eqref{eq:difference-factorization} holds. Thus, for $n\ge4$, the proof of Proposition~\ref{prop:difference-differential} applies unchanged and gives equation~\eqref{eq:difference-fourth}; the cases $n=2,3$ give $d=2,6$.

Let $n\ge4$, and let $\ell\ge2$ be least with $p_\ell(A)\ne0$. The coefficient computation in the proof of Theorem~\ref{thm:difference}, before using that $\ell$ is odd, gives $(6-\ell)(n+1)=12$. Hence $(n,\ell)=(5,4)$ or $(11,5)$. The final coefficient comparison in that proof excludes $(11,5)$, so $F(X)=X^5+\xi X+\eta$ with $\xi\ne0$. Substitution into equation~\eqref{eq:difference-fourth} gives $3600\eta X=0$. Since $p>25$, we obtain $\eta=0$, and hence $A=\{0\}\cup\gamma K$, where $K<\Fps$ has order $4$ and $\gamma$ is one of the roots of $X^4+\xi$.
Since $0\in A$ and $A-A\subseteq H\cup\{0\}$, we have $\gamma K\subseteq H$, so $\gamma\in H$ and $K\subseteq H$. If $i\in K$ has order $4$, then $\gamma(1\pm i)\in H$, whence $1\pm i\in H$ and therefore $2\in H$. Since $d=20$, it follows that $2^{20}=1$ in $\Fp$. Thus $p\mid2^{20}-1$ and $p\equiv1\pmod{20}$. As $2^{20}-1=3\cdot5^2\cdot11\cdot31\cdot41$, we obtain $p=41$. This proves~(i).

Now suppose that $q=d-2$. Let $\Psi=\Psi_{-A}$ be the polynomial from Proposition~\ref{prop:common-factorization}(i). Then $\Psi=F^{n-1}R$, where $R$ is monic of degree $2$ and has no root in $A$. Since $p_1(A)=0$, equation~\eqref{eq:deletion-top-coefficients}, applied with $S=-A$, gives $[X^{d-1}]\Psi=0$. Also $[X^{q-1}]F^{n-1}=0$, so $R(X)=X^2+\rho$ for some $\rho\in\Fp$.

Fix $b\in A$. Since $R(b)\ne0$, the coefficient of $Y^{n-1}$ in $\Psi(b+Y)=F(b+Y)^{n-1}R(b+Y)$ is nonzero. Dividing the $Y^n$-coefficient identity by the $Y^{n-1}$-coefficient identity, and using $(b-a)^d=1$ for $a\ne b$ together with equation~\eqref{eq:difference-local-pf}, gives
\[
 \frac{d}{2n}\frac{F''(b)}{F'(b)}=\frac{n-1}{2}\frac{F''(b)}{F'(b)}+\frac{R'(b)}{R(b)}.
\]
Since $d=n(n-1)+2$ and $R=X^2+\rho$, we obtain
\begin{equation}\label{eq:difference-near-critical-ode}
 (X^2+\rho)F''-2nXF'+n(n+1)F=0.
\end{equation}
Indeed, the left-hand side vanishes at every element of $A$ and has degree at most $n-1$.

The monic degree-$n$ solution of equation~\eqref{eq:difference-near-critical-ode} is unique. Indeed, the difference of two such solutions would be a nonzero solution of degree $m<n$, whose leading coefficient is multiplied by $(m-n)(m-n-1)\ne0$, contradicting equation~\eqref{eq:difference-near-critical-ode}; here $p>n+1$. Hence $\rho\ne0$, since for $\rho=0$ the unique monic solution is $X^n$, contrary to the squarefreeness of $F$. Choose $\sigma\in\overline{\Fp}$ with $\sigma^2=-\rho$. Direct substitution to equation~\eqref{eq:difference-near-critical-ode} and uniqueness give
\begin{equation}\label{eq:difference-near-critical-form}
 F(X)=\frac{(X+\sigma)^{n+1}-(X-\sigma)^{n+1}}{2(n+1)\sigma}.
\end{equation}
Thus $F(-X)=(-1)^nF(X)$ and $A=-A$.

Let $\mu_{n+1}\subseteq\overline{\Fp}^{\,\times}$ denote the group of $(n+1)$-st roots of unity. By equation~\eqref{eq:difference-near-critical-form}, the roots of $F$ are
\[
 x_\zeta=\sigma\frac{\zeta+1}{\zeta-1},\qquad \zeta\in\mu_{n+1}\setminus\{1\},
\]
and this parametrization is injective. Since these are precisely the roots of $F=f_A$, we have $x_\zeta\in\Fp$. Moreover, $\sigma^2\in\Fp$ gives $\sigma^p=\pm\sigma$. Hence
\[
 x_\zeta^p=
 \begin{cases}
 x_{\zeta^p},&\sigma^p=\sigma,\\
 x_{\zeta^{-p}},&\sigma^p=-\sigma.
 \end{cases}
\]
Since $x_\zeta^p=x_\zeta$, injectivity gives respectively $\zeta^p=\zeta$ for every $\zeta\in\mu_{n+1}$ or $\zeta^p=\zeta^{-1}$ for every $\zeta\in\mu_{n+1}$. Therefore $n+1\mid p-1$ or $n+1\mid p+1$.

Finally, $d=n(n-1)+2\equiv4\pmod{n+1}$ and $p-1=rd$. Hence $n+1\mid4r$ or $n+1\mid4r+2$, so $n+1\le4r+2$ and $d\le(4r+1)(4r)+2=16r^2+4r+2$. This proves~(ii).
\end{proof}

\begin{proposition}\label{prop:difference-deletion}
Let $p$ be an odd prime, let $H<\Fps$ be a proper multiplicative subgroup, and let $h\in H$. Suppose that $A\subseteq\Fp$, $|A|\ge2$, satisfies $A-A=(H\setminus\{h,-h\})\cup\{0\}.$ Then $|H|=4$.
\end{proposition}

\begin{proof}
Put $n=|A|$, $d=|H|$, and $q=n(n-1)$. Since $A-A$ is symmetric and contains a nonzero element of $H$, we have $-1\in H$, so $d$ is even. Proposition~\ref{prop:common-factorization}(i) gives $q\le d$, while $A-A$ has $d-2$ nonzero elements. Hence $d-2\le q\le d$, and parity gives $q\in\{d-2,d\}$.

Suppose $q=d$. Translate $A$ so that $p_1(A)=0$, and put $F=f_A$. With $\Psi=\Psi_{-A}$, Proposition~\ref{prop:common-factorization}(i) gives $\Psi=F^{n-1}$. Equation~\eqref{eq:deletion-top-coefficients}, applied with $S=-A$, and Lemma~\ref{lem:newton} show, by comparing the coefficients of $X^{d-2}$, that $\frac{n^2(n-1)}{2(n+1)}p_2(A)=0$. Since $d+n-1<p$, the coefficient is nonzero, so $p_2(A)=0$. On the other hand, the number of representations of $x$ as a difference $a-b$ equals that of $-x$. Since the $d$ ordered pairs $(a,b)\in A^2$ with $a\ne b$ give only the $d-2$ nonzero differences in $H\setminus\{h,-h\}$, there is $z\in H\setminus\{h,-h\}$ such that $z$ and $-z$ each have two representations, while every other nonzero difference has one. Using $p_1(A)=0$ and equation~\eqref{eq:subgroup-power-sum},
\[
 2np_2(A)=\sum_{a,b\in A}(a-b)^2=2(z^2-h^2)\ne0,
\]
a contradiction. Thus $q=d-2$.

Lemma~\ref{lem:difference-near-critical}(ii) now gives, after a translation, $A=-A$. Since $q=d-2=|A-A|-1$, every nonzero difference has a unique representation $a-b$ with $(a,b)\in A^2$. If $n>2$, choose distinct $a,b\in A$ with $b\ne-a$. Then $a-b=(-b)-(-a)$ gives two distinct representations of the same nonzero difference, a contradiction. Hence $n=2$ and $d=q+2=4$.
\end{proof}

\begin{proof}[Proof of Theorem~\ref{thm:difference-stability}]
If $A-A=\{0\}$, all three conclusions are immediate. Otherwise choose $x\in(A-A)\setminus\{0\}$. Since $x,-x\in H$, we have $-1\in H$, so both $d$ and $q:=n(n-1)$ are even.

For~(i), the elements of $H$ missing from $A-A$ occur in pairs $\{h,-h\}$. Since $d>6$, the case with no missing elements is excluded by Theorem~\ref{thm:difference}, and the case with exactly one missing pair is excluded by Proposition~\ref{prop:difference-deletion}. Thus at least four elements of $H$ are missing, and $|A-A|\le d-3$.

For~(ii), Proposition~\ref{prop:common-factorization}(i) gives $q\le d$. If $q=d$, Lemma~\ref{lem:difference-near-critical}(i) and $d>6$ force $(p,d)=(41,20)$. Outside this case, parity gives $q\le d-2$.

For~(iii), the hypothesis $d>16r^2+4r+2$ excludes $(p,d)=(41,20)$, so~(ii) gives $q\le d-2$. Equality is impossible by Lemma~\ref{lem:difference-near-critical}(ii), and parity therefore gives $q\le d-4$. Finally, if $d>3p^{2/3}$, then $r<p^{1/3}/3$. Here necessarily $p\ge17$, and 
\[16r^2+4r+2<\frac{16}{9}p^{2/3}+\frac43p^{1/3}+2<3p^{2/3}<d.\qedhere\]
\end{proof}

We next apply the stability theorem to the clique number of Paley graphs. For a prime $p \equiv 1 \pmod 4$, write $\omega(P_p)$ for the clique number of the Paley graph $P_p$ over $\Fp$, that is, the largest possible size of $A \subseteq \Fp$ such that $A-A \subseteq \mathcal{R}_p \cup \{0\}$. Estimating such clique number is a long-standing open question in analytic number theory and arithmetic combinatorics. Little progress beyond the trivial upper bound \(\sqrt p\) was known, until Hanson and Petridis~\cite{HP} famously proved the sharper bound $\frac{\sqrt{2p-1}+1}{2}$. See \cite{Y25a} for further discussion of clique-number bounds for Paley graphs and their generalizations. 

Kalmynin~\cite[Definition~2]{Kalmynin} calls a pair $(B,C)$ $d$-critical if $B+C\subseteq H\cup\{0\}$ and $|B||C|=d+|(-B)\cap C|$, where $|H|=d$. For $C=-B$, this condition is exactly $|B|(|B|-1)=d$. His classification of this critical difference-set case~\cite[Theorem~4]{Kalmynin}, which is part of his proof of the Lev--Sonn conjecture, shows that, apart from $d\in\{2,6\}$, the only possibility is $(p,d)=(41,20)$; in that exceptional case the containment $B-B\subseteq H\cup\{0\}$ is strict. Consequently, if $p>41$ and $A$ is a clique in $P_p$, then $|A|(|A|-1)\ne|\mathcal R_p|$. Since both quantities are even, one obtains $\omega(P_p)\le\frac{1+\sqrt{2p-9}}2.$ Theorem~\ref{thm:difference-stability}(ii) gives an independent differential proof of this critical-pair consequence. Lemma~\ref{lem:difference-near-critical}(ii) goes one step further by analyzing the first near-critical case $|A|(|A|-1)=|\mathcal R_p|-2$. We prove the following improvement. 
\begin{corollary}\label{cor:paley-clique}
If $p>41$ is prime and $p\equiv1\pmod4$, then $ \omega(P_p)\le\frac{1+\sqrt{2p-17}}2.$
\end{corollary}

\begin{proof}
Let $A\subseteq\Fp$ be a clique of size $n$, and put $H=\mathcal R_p$ and $d=|H|=(p-1)/2$. Then $A-A\subseteq H\cup\{0\}$. By Theorem~\ref{thm:difference-stability}(ii), $q:=n(n-1)\le d-2$.

Suppose $q=d-2$. Lemma~\ref{lem:difference-near-critical}(ii), with $r=2$, gives $n+1\mid p-1$ or $n+1\mid p+1$. Since $p=2n(n-1)+5$, we obtain $n+1\mid8$ or $n+1\mid10$. Thus $n\in\{3,4,7,9\}$, and $p>41$ leaves only $(n,p)=(7,89)$ and $(9,149)$. A direct exact computation via SageMath gives $\omega(P_{89})=5$ and $\omega(P_{149})=7$, so neither case can occur. Hence $q<d-2$, and parity gives $q\le d-4=(p-9)/2$. Solving the quadratic inequality gives the claimed bound.
\end{proof}

\section{Restricted sumsets}\label{sec:restricted}

In this section, we prove Theorems~\ref{thm:restricted} and~\ref{thm:restricted-zero}. Related auxiliary-polynomial arguments for restricted sumsets were developed in \cite[Sections~3 and 4]{Y25} under additional assumptions. The present section removes these assumptions and gives an unconditional treatment of the two restricted-sumset equations considered here.

Throughout this section, assume
\[
 A\rsum A\subseteq H\cup\{0\},
\]
where $H<\Fps$ is proper. Put $d=|H|$ and $n=|A|=2t+\varepsilon$, where $t\ge1$ and $\varepsilon\in\{0,1\}$, and write $F=f_A$ and $c_a=1/F'(a)$ for $a\in A$.
Recall from Section~\ref{sec:Lagrange} that for a finite set $S\subseteq\Fp$,
\[
 \Phi_{S;L,k}(X)=\sum_{s\in S}\frac{(X+s)^L(X-s)^k}{f_S'(s)}.
\]
Moreover, fixing $b\in A$ gives $n-1$ distinct restricted sums $a+b$ with $a\in A\setminus\{b\}$, and hence $n-1\le |A\rsum A|\le d+1$.
Since $H$ is proper, $2d\le p-1$, and
\[
 d+t-1+\varepsilon\le d+\frac{n-1}{2}\le\frac{3d+1}{2}<p.
\]
Thus every binomial coefficient that we divide by below has upper index less than $p$, and hence is nonzero in $\Fp$. We shall also use that $n<p$ whenever $p>3$, since $n\le d+2\le(p+3)/2<p$.

For $b\in A$ and integers $j\ge0$ and $s\ge1$, put
\[
 M_{j,s}(b)=\sum_{\substack{a\in A\\a\ne-b}}c_a\frac{(b-a)^j}{(b+a)^s}.
\]

\subsection{Restricted-sumset interpolation}
The following proposition records the resulting local multiplicity and the higher-order coefficients used below.

\begin{proposition}\label{prop:restricted-local-expansion}
Let $m\ge0$ and $\delta\in\{0,1\}$ satisfy
\[
 1\le 2m+\delta\le n-1.
\]
Set
\[
 \Psi(X)=\Phi_{A;d+m,m+\delta}(X)-\eta,\qquad\eta=
 \begin{cases}
 (-1)^{m+\delta},& n=2m+\delta+1,\\
 0,& n>2m+\delta+1.
 \end{cases}
\]
Then, for every $b\in A$, the following hold.
\begin{enumerate}
\item[(i)] We have $[Y^j]\Psi(b+Y)=0$ for $0\le j<m$. Moreover, writing $\mathbf 1_E$ for the indicator of a condition $E$, we have
\[
 [Y^m]\Psi(b+Y)=\mathbf 1_{\delta=0}\,c_b(2b)^m\bigl((2b)^d-1\bigr)-\mathbf 1_{-b\in A}\binom{d+m}{m}c_{-b}(2b)^{m+\delta}.
\]

\item[(ii)] Suppose $b\ne0$. If $r\ge1$ and either $-b\notin A$ or $r<d$, then
\begin{align}
 [Y^{m+r}]\Psi(b+Y)=&\binom{d+m}{r-\delta}c_b(2b)^{m+\delta-r}\bigl((2b)^d-1\bigr)-\mathbf 1_{-b\in A}\binom{d+m}{m}\binom{m+\delta}{r}c_{-b}(2b)^{m+\delta-r} \notag\\
&+\sum_{s=\max\{1,r-m-\delta\}}^{\min\{r,d\}}\binom{d+m}{m+s}\binom{m+\delta}{r-s}M_{m+\delta-r+s,s}(b).
\label{eq:restricted-coefficient-reduction}
\end{align}
\end{enumerate}
\end{proposition}

\begin{proof}
For $0\leq q\le m$ and $0\leq \ell\le m+\delta$, set
\[
 P_{q,\ell}(X)=(X+b)^{m-q}(b-X)^{m+\delta-\ell} \in \Fp[X].
\]
For $0 \leq q\leq d+m$ and $0\leq \ell\le m+\delta$, set
\[
 I_{q,\ell}=\sum_{a\in A}c_a(a+b)^{d+m-q}(b-a)^{m+\delta-\ell}.
\]
By equation~\eqref{eq:phi-local-expansion},
\begin{equation}\label{eq:phi-expansion}
 [Y^j]\Phi_{A;d+m,m+\delta}(b+Y)=\sum_{q+\ell=j}\binom{d+m}{q}\binom{m+\delta}{\ell}I_{q,\ell}.    
\end{equation}

Suppose first that $0\le j<m$. Then $q,\ell<m$. For $a\ne b,-b$, we have $a+b\in H$, and hence $(a+b)^d=1$. At $a=b,-b$, both the original summand and $P_{q,\ell}(a)$ vanish. Therefore
\[
 I_{q,\ell}=\sum_{a\in A}c_aP_{q,\ell}(a).
\]
Since $\deg P_{q,\ell}=2m+\delta-j\le n-1$, Lemma~\ref{lem:lagrange} gives
\[
 I_{q,\ell}=[X^{n-1}]P_{q,\ell}(X).
\]
This is zero unless $j=0$ and $n=2m+\delta+1$. In that case $q=\ell=0$ and the leading coefficient of $P_{0,0}$ is $(-1)^{m+\delta}=\eta$. Thus $[Y^j]\Psi(b+Y)=0$ for $0\le j<m$.

Now let $j=m$. Again $\deg P_{q,\ell}=m+\delta\le n-1$, so by Lemma~\ref{lem:lagrange}, $\sum_{a\in A}c_aP_{q,\ell}(a)=0$, except when $(m,\delta,n)=(0,1,2)$, where it equals $\eta$ and is canceled by the constant term in $\Psi$. Evaluating the terms at $a=b,-b$ separately gives
\[
 I_{q,\ell}=\sum_{a\in A}c_aP_{q,\ell}(a)+\mathbf 1_{\delta=0,\ q=0}\,c_b(2b)^m\bigl((2b)^d-1\bigr)-\mathbf 1_{-b\in A,\ q=m}\,c_{-b}(2b)^{m+\delta}.
\]
Substituting this into equation~\eqref{eq:phi-expansion} with $j=m$ proves part~(i).

For part~(ii), let $r\ge1$ satisfy the stated hypotheses and set $j=m+r$. Since the constant $\eta$ does not contribute to the coefficient of $Y^{m+r}$,
\begin{equation}\label{eq:coff-phi}
 [Y^{m+r}]\Psi(b+Y)=\sum_{q+\ell=m+r}\binom{d+m}{q}\binom{m+\delta}{\ell}I_{q,\ell}.
\end{equation}
Suppose first that $r-\delta\le q\le m$. Then
\[
  P_{q,\ell}(X)=(X+b)^{m-q}(b-X)^{q+\delta-r}
\]
and $\deg P_{q,\ell}=m+\delta-r<n-1$, so Lemma~\ref{lem:lagrange} gives $\sum_{a\in A}c_aP_{q,\ell}(a)=0$. For $a\in A\setminus\{b,-b\}$, we have $(a+b)^{d+m-q}(b-a)^{m+\delta-\ell}=P_{q,\ell}(a)$. Moreover, $P_{q,\ell}(b)$ can be nonzero only when $q=r-\delta$, while $P_{q,\ell}(-b)$ can be nonzero only when $q=m$. Hence
\begin{equation}\label{eq:1}
 I_{q,m+r-q}=\mathbf 1_{q=r-\delta}\,c_b(2b)^{m+\delta-r}\bigl((2b)^d-1\bigr)-\mathbf 1_{-b\in A,\ q=m}\,c_{-b}(2b)^{m+\delta-r}.
\end{equation}

Now suppose $q>m$. Write $q=m+s$, so that $\ell=r-s$. The binomial coefficients in equation~\eqref{eq:coff-phi} can be nonzero only when $\max\{1,r-m-\delta\}\le s\le\min\{r,d\}$. For $a\in A\setminus\{b,-b\}$, the subgroup condition gives
\[
 (a+b)^{d-s}(b-a)^{m+\delta-r+s}=\frac{(b-a)^{m+\delta-r+s}}{(a+b)^s}.
\]
If $-b\in A$, then $s\le r<d$, so the left-hand side vanishes at $a=-b$. At $a=b$, the two sides of the preceding identity both vanish unless $m+\delta-r+s=0$. In that case, their values are $(2b)^{d-s}$ and $(2b)^{-s}$, respectively. Hence
\begin{equation}\label{eq:2}
 I_{m+s,r-s}=M_{m+\delta-r+s,s}(b)+\mathbf 1_{m+\delta-r+s=0}\,c_b(2b)^{-s}\bigl((2b)^d-1\bigr).
\end{equation}
Since $m+\delta-r+s=0$ is equivalent to $q=r-\delta$, the $a=b$ term occurs precisely at $q=r-\delta$, whether this index lies in the range $q\le m$ or $q>m$. The $a=-b$ term occurs only at $q=m$, while the remaining terms give the sum involving $M_{m+\delta-r+s,s}(b)$. 

Substituting equations~\eqref{eq:1} and \eqref{eq:2} into equation~\eqref{eq:coff-phi} proves part~(ii).
\end{proof}

We evaluate the sums $M_{j,s}(b)$ below using Lemma~\ref{lem:weighted-partial-fractions} and formal power-series expansions at $Y=0$, isolating the $a=-b$ term when necessary.

\subsection{The equation \texorpdfstring{$H=A\rsum A$}{H equals a restricted sumset}}

Assume first that $A\rsum A=H$.
We first establish the auxiliary vanishing identity. The same local coefficient calculation also yields the differential relation used below and, one order further, the second-order relation needed at the end of the proof.

\begin{proposition}\label{prop:restricted-differential}
Assume $d\ge3$. Then
\begin{equation}\label{eq:restricted-vanishing}
 \sum_{a\in A}c_a(X+a)^{d+t-1}(X-a)^t = (1-\varepsilon)(-1)^t.
\end{equation}
Moreover, $\{2a:a\in A\}\cap H=\varnothing$, and
\begin{equation}\label{eq:restricted-first}
 2XF'(-X)+(2t+1)F(-X)=(-1)^{n+1}(n-1+\varepsilon)F(X).
\end{equation}
If $t\ge2$, then for every $b\in A\setminus\{0\}$,
\begin{equation}\label{eq:restricted-second-local}
 (d-2)(2b)^2F''(-b)=2t\bigl((2t+1)d-3t\bigr)F(-b).
\end{equation}
\end{proposition}
\begin{proof}
Put
\[
 \Psi=\Phi_{A;d+t-1,t}-(1-\varepsilon)(-1)^t.
\]
Let $b\in A$. Apply Proposition~\ref{prop:restricted-local-expansion}(i) with $m=t-1$ and $\delta=1$. It shows that $[Y^j]\Psi(b+Y)=0$ for $0\leq j<t-1$, and the formula for the next coefficient gives $[Y^{t-1}]\Psi(b+Y)=0$: if $b\ne0$, then $-b\notin A$ since $0\notin H$, while if $b=0$, the second term in Proposition~\ref{prop:restricted-local-expansion}(i) contains the factor $(2b)^t=0$. Thus \(b\) is a root of \(\Psi\) of multiplicity at least \(t\). Hence $F^t\mid\Psi$. By Lemma~\ref{lem:global-interpolation}, $\deg\Psi\le d-\varepsilon$. Note that $d=|A\rsum A|\le\binom n2$ and $\binom n2-\varepsilon<nt=\deg F^t$. Thus, $\Psi=0$, proving equation~\eqref{eq:restricted-vanishing}.

Fix $b\in A\setminus\{0\}$. Then $-b\notin A$. Set
\[
 D=d+t-1,\qquad B=\binom Dt, \qquad u=2b\frac{F'(-b)}{F(-b)},\qquad v=\frac{(2b)^2F''(-b)}{2F(-b)}.
\]
For $0\le j\le n-1$, Lemma~\ref{lem:weighted-partial-fractions} gives
\begin{equation}\label{eq:511}
 -\frac{(2b+Y)^j}{F(-b-Y)} \equiv M_{j,1}(b)-M_{j,2}(b)Y+M_{j,3}(b)Y^2 \pmod{Y^3}.
\end{equation}
Since $F(-b)\ne0$, Taylor expansion at $Y=0$ for $F(-b-Y)$ gives
\begin{equation}\label{eq:512}
 \frac{F(-b-Y)}{F(-b)} \equiv 1-\frac{u}{2b}Y+\frac{v}{(2b)^2}Y^2 \pmod{Y^3}.
\end{equation}
It follows from equations~\eqref{eq:511} and the inversion of equation~\eqref{eq:512} by Lemma~\ref{lem:coefficient-expansions} that
\[
F(-b)(M_{j,1}(b)-M_{j,2}(b)Y+M_{j,3}(b)Y^2)\equiv -\bigg(1+\frac{u}{2b}Y+\frac{u^2-v}{(2b)^2}Y^2\bigg) (2b+Y)^j \pmod{Y^3}.
\]
Comparing coefficients of $1,Y,Y^2$ on both sides gives
\begin{equation}\label{eq:Mj}
 M_{j,1}(b)=-\frac{(2b)^j}{F(-b)}, \quad M_{j,2}(b)=\frac{(2b)^{j-1}}{F(-b)}(j+u), \quad M_{j,3}(b)=\frac{(2b)^{j-2}}{F(-b)}\left(v-u^2-ju-\binom j2\right).
\end{equation}

Since $b\ne0$ and $-b\notin A$, Proposition~\ref{prop:restricted-local-expansion}(ii) applies with $m=t-1$ and $\delta=1$ for every $r\ge1$. As $\Psi=0$, the left-hand side of equation~\eqref{eq:restricted-coefficient-reduction} vanishes. Taking $r=1$ yields
\[
 0=c_b(2b)^{t-1}\bigl((2b)^d-1\bigr)+B M_{t,1}(b).
\]
Substituting the formula for $M_{t,1}(b)$ in equation~\eqref{eq:Mj} gives
\begin{equation}\label{eq:restricted-order-t}
 c_bF(-b)\frac{(2b)^d-1}{2b}=B.
\end{equation}
Since $B\neq 0$, it follows that $(2b)^d\ne1$ and thus $2b\notin H$. Since \(0\notin H\), this also proves $\{2a:a\in A\}\cap H=\varnothing$. For $r=2$, equation~\eqref{eq:restricted-coefficient-reduction} becomes
\[
 0=Dc_b(2b)^{t-2}\bigl((2b)^d-1\bigr)+tB M_{t-1,1}(b)+\binom D{t+1}M_{t,2}(b).
\]
Substituting equation~\eqref{eq:restricted-order-t} and the formulas for $M_{t-1,1}(b)$ and $M_{t,2}(b)$ in equation~\eqref{eq:Mj}, and then dividing by the common nonzero factor $B(2b)^{t-1}/F(-b)$, gives
\[
 0=(d-1)+\frac{d-1}{t+1}(t+u)=\frac{d-1}{t+1}(u+2t+1).
\]
Hence $u=-(2t+1)$. We have therefore obtained $2bF'(-b)+(2t+1)F(-b)=0$ for every $b\in A\setminus\{0\}$.

For $b=0\in A$, the relation $2bF'(-b)+(2t+1)F(-b)=0$ follows from $F(0)=0$. Hence every root of $F$ is a root of
\[
 R(X)=2XF'(-X)+(2t+1)F(-X).
\]
Thus $F\mid R$. Since $\deg R\le n=\deg F$, we have $R=\lambda F$ for some $\lambda\in\Fp$. Comparing the coefficients of $X^n$ gives $\lambda=(-1)^{n+1}(n-1+\varepsilon)$, which proves equation~\eqref{eq:restricted-first}.

Let $b\in A \setminus \{0\}$. If $t\ge2$, taking $r=3$ in equation~\eqref{eq:restricted-coefficient-reduction} gives
\[
0=\binom D2 c_b(2b)^{t-3}\bigl((2b)^d-1\bigr)+\binom t2 B M_{t-2,1}(b)+t\binom D{t+1}M_{t-1,2}(b)+\binom D{t+2}M_{t,3}(b).
\]
Substituting equation~\eqref{eq:restricted-order-t} and the formulas for $M_{j,s}(b)$ in equation~\eqref{eq:Mj}, and dividing by $B(2b)^{t-2}/F(-b)$, gives
\[
0=\binom D2-\binom t2+\frac{t(d-1)}{t+1}(u+t-1)+\frac{(d-1)(d-2)}{(t+1)(t+2)}\left(v-u^2-tu-\binom t2\right).
\]
Using $u=-(2t+1)$, this simplifies to $(d-2)v=t\bigl((2t+1)d-3t\bigr)$, which is equation~\eqref{eq:restricted-second-local}.
\end{proof}

We next determine $F$ from the differential identity and then determine $d$.

\begin{proposition}\label{prop:restricted-structure}
Assume $d\ge3$. Then, for some $\xi\in\Fps$,
\begin{equation}\label{eq:restricted-form}
 F(X)=X^n+\xi X^{1-\varepsilon}.
\end{equation}
Moreover, $d=\binom n2$, and every element of $H$ has a unique representation as an unordered sum of two distinct elements of $A$.
\end{proposition}

\begin{proof}
By Proposition~\ref{prop:restricted-differential}, equation~\eqref{eq:restricted-first} holds. Write $F(X)=\sum_j a_jX^j$. Comparing the coefficient of $X^j$ gives
\[
 \begin{cases}
  2(j-n)a_j=0,&(-1)^{n+j}=1,\\
  -2(j-1+\varepsilon)a_j=0,&(-1)^{n+j}=-1.
 \end{cases}
\]
Since $n<p$, these factors can vanish only when $j=n$ or $j=1-\varepsilon$. Hence $F(X)=X^n+\xi X^{1-\varepsilon}$ for some $\xi\in\Fps$.

Put $m=t-1+\varepsilon$, so that $n=2m+2-\varepsilon$. Since $d\ge3$, we have $m\ge1$.
Equation~\eqref{eq:restricted-form} shows that $0\in A$ exactly when $\varepsilon=0$. Let $T=A$ if $\varepsilon=1$, and $T=A\setminus\{0\}$ if $\varepsilon=0$. Then $|T|=2m+1$ and $0\notin T$. Set
\[
 \Omega(X)=\Phi_{T;d+m,m}(X)-(-1)^m.
\]
Lemma~\ref{lem:global-interpolation} gives $\deg\Omega\le d$.

For every $b\in T$, we have $-b\notin T$, since $T\rsum T\subseteq H$ and $0\notin H$. 
Moreover, Proposition~\ref{prop:restricted-differential} gives $2b\notin H$.
Thus Proposition~\ref{prop:restricted-local-expansion}(i), applied with \(A=T\) and \(\delta=0\), shows that \(b\) is a root of \(\Omega\) of multiplicity exactly \(m\): \([Y^j]\Omega(b+Y)=0\) for \(0\le j<m\), while \([Y^m]\Omega(b+Y)\ne0\) since \(b\ne0\) and \((2b)^d\ne1\).

If $\varepsilon=0$, then $0\in A$, so $T\subseteq H$ because $b=0+b\in A\rsum A=H$ for every $b\in T$. Hence the final assertion of Lemma~\ref{lem:global-interpolation} gives $X^{2m+1}\mid\Omega$. Thus the roots in $T$ contribute total multiplicity $m(2m+1)$, and when $\varepsilon=0$ the root at $0$ contributes at least $2m+1$ more. In either case, the total multiplicity is at least $(m+1-\varepsilon)(2m+1)=\binom n2$. Consequently,
\[
 \binom n2\le\deg\Omega\le d\le\binom n2.
\]
Equality holds throughout. Thus $d=\binom n2$, and the $\binom n2$ unordered pairs of distinct elements of $A$ give exactly the $d$ elements of $H$, so every element of $H$ has a unique representation.
\end{proof}

\begin{proof}[Proof of Theorem~\ref{thm:restricted}]
Clearly $|A\rsum A|\ne2$. Thus, we may assume $d\ge3$. By Proposition~\ref{prop:restricted-structure}, $d=\binom n2$ and $F(X)=X^n+\xi X^{1-\varepsilon}$ for some $\xi\in\Fps$. If $t=1$, then $d=\binom n2\ge3$ forces $n=3$ and $d=3$, so assume $t\ge2$ and fix $b\in A\setminus\{0\}$. Proposition~\ref{prop:restricted-differential} gives
\begin{equation}\label{eq:db}
 (d-2)(2b)^2F''(-b)=2t\bigl((2t+1)d-3t\bigr)F(-b).
\end{equation}

Since $F(b)=0$ and $n=2t+\varepsilon$, we have $F(-b)=2(-1)^\varepsilon b^n$ and $F''(-b)=(-1)^\varepsilon n(n-1)b^{n-2}$. Substituting these into equation~\eqref{eq:db} and canceling $4(-1)^\varepsilon b^n$ gives
\begin{equation}\label{eq:db-simplified}
 (d-2)n(n-1)=t\bigl((2t+1)d-3t\bigr).
\end{equation}

If $n=2t$, then $d=t(2t-1)$, and equation~\eqref{eq:db-simplified} simplifies to $2t(t-2)(2t^2-1)=0$. Since $t<p$ and $0<2t^2-1<2d+1\le p$, the other factors are nonzero in $\Fp$, so $t=2$ and $(n,d)=(4,6)$.

If $n=2t+1$, then $d=t(2t+1)$, and equation~\eqref{eq:db-simplified} becomes $4t(t-1)(t+1)^2=0$. Since $t\ge2$ and $t+1<p$, every factor is nonzero in $\Fp$, a contradiction. Together with the case $t=1$ treated above, this gives $d\in\{1,3,6\}$.
\end{proof}

\begin{remark}\label{rem:restricted-classification}
The proof also gives the full classification. If $|H|=1$, then $H=\{1\}$ and $A=\{x,1-x\}$ for some $x\ne\frac12$. If $|H|=3$, then $(n,d)=(3,3)$ and equation~\eqref{eq:restricted-form} gives $F(X)=X^3+\xi$. Choosing $\gamma\in A$, we have $A=\gamma K$ for the subgroup $K$ of order $3$, and hence $A\rsum A=-A=H$, so $A=-H$. If $|H|=6$, then $(n,d)=(4,6)$ and equation~\eqref{eq:restricted-form} gives $A=\{0\}\cup\gamma K$ for a subgroup $K$ of order $3$. Since $0\in A$, we have $\gamma K\subseteq H$. As $1\in K$, this gives $\gamma\in H$, and hence $K\subseteq H$. Since $K$ has index $2$ in $H$, we have $\gamma K=K$ or $-K$. Thus $A=\{0\}\cup K$ or $A=\{0\}\cup(-K)$. Conversely, each of these sets gives $H=A\rsum A$.
\end{remark}

\subsection{The equation \texorpdfstring{$H\cup\{0\}=A\rsum A$}{H union zero equals a restricted sumset}}

Assume now that $A\rsum A=H\cup\{0\}$.
We first determine the symmetry of $A$, the value of $d$, and the auxiliary vanishing identity.

\begin{proposition}\label{prop:restricted-zero-structure}
Assume $p>3$ and $A\rsum A=H\cup\{0\}$. Then $A=-A$, every element of $H$ has a unique representation as an unordered sum of two distinct elements of $A$, and
\begin{equation}\label{eq:restricted-zero-size}
 d=\binom n2-t=2t(t-1+\varepsilon).
\end{equation}
Moreover, with $m=t-1+\varepsilon$,
\begin{equation}\label{eq:restricted-zero-vanishing}
 \sum_{a\in A}c_a(X+a)^{d+m}(X-a)^m=\varepsilon(-1)^m.
\end{equation}
\end{proposition}

\begin{proof}
We treat the two parities separately.

\emph{Even $n$.}
In this case, $n=2t$ and $m=t-1$. Necessarily $t\ge2$, since otherwise $|A\rsum A|=1$. Let $h$ be the number of nonzero unordered pairs $\{a,-a\}\subseteq A$. Since $0\in A\rsum A$, we have $1\le h\le t$. Put
\[
 \Psi=\Phi_{A;d+t-1,t}-(-1)^t.
\]
Lemma~\ref{lem:global-interpolation} gives $\deg\Psi\le d$.
Proposition~\ref{prop:restricted-local-expansion}(i), with $m=t-1$ and $\delta=1$, gives multiplicity at least $t-1$ at every $b\in A$ and it also gives 
\[
 [Y^{t-1}]\Psi(b+Y)=-\mathbf 1_{-b\in A}\binom{d+t-1}{t-1}c_{-b}(2b)^t,
\]
which is nonzero exactly when $b\ne0$ and $-b\in A$. Thus the $2h$ elements in nonzero opposite pairs have multiplicity exactly $t-1$, while the remaining $2t-2h$ elements have multiplicity at least $t$. Since $h\geq 1$, $\Psi$ is not the zero polynomial. These roots contribute total multiplicity at least $2h(t-1)+(2t-2h)t=2t^2-2h$, so $2t^2-2h\le\deg\Psi$. Since exactly $h$ unordered pairs have zero sum,
\[
 \deg \Psi\leq d=|(A \rsum A)\setminus \{0\}|\le\binom{2t}{2}-h=2t^2-t-h.
\]
Comparing the lower and upper bounds on $\deg \Psi$ gives $2t^2-2h\le2t^2-t-h$, hence $h\ge t$. Since $h\le t$, we must have $h=t$, and consequently equality holds throughout the chain. Thus every element of $A$ occurs in an opposite pair; hence $A=-A$, and $d=2t(t-1)$, proving equation~\eqref{eq:restricted-zero-size}. There are then exactly $\binom{2t}{2}-t=2t(t-1)=d$ unordered pairs with nonzero sum, and their sums cover $H$. Therefore every element of $H$ has a unique representation.

Finally, Proposition~\ref{prop:restricted-local-expansion}(i), now with $\delta=0$, gives $F^m\mid\Phi_{A;d+t-1,t-1}$. By Lemma~\ref{lem:global-interpolation},
\[
 \deg\Phi_{A;d+t-1,t-1}\le d-1<2t(t-1)=\deg F^m.
\]
Hence $\Phi_{A;d+t-1,t-1}=0$, proving equation~\eqref{eq:restricted-zero-vanishing} when $\varepsilon=0$.

\emph{Odd $n$.} In this case, we have $n=2t+1$ and $m=t$. Put
\[
 \Psi=\Phi_{A;d+t,t}-(-1)^t.
\]
Proposition~\ref{prop:restricted-local-expansion}(i), with $m=t$ and $\delta=0$, gives $F^t\mid\Psi$, while Lemma~\ref{lem:global-interpolation} gives $\deg\Psi\le d$. Since
\[
 d+1=|A\rsum A|\le\binom{2t+1}{2}=t(2t+1),
\]
we have $\deg\Psi<\deg F^t$. Thus $\Psi=0$, proving equation~\eqref{eq:restricted-zero-vanishing} when $\varepsilon=1$.

Suppose $b\ne0$ and $-b\notin A$. Proposition~\ref{prop:restricted-local-expansion}(i), together with $\Psi=0$, gives $(2b)^d=1$. Under these assumptions, part~(ii) applies with $m=t$, $\delta=0$, and $r=1$. As $\Psi=0$, the left-hand side of equation~\eqref{eq:restricted-coefficient-reduction} vanishes. The first term on the right vanishes because $(2b)^d=1$, and the second is absent because $-b\notin A$, so
\[
 0=\binom{d+t}{t+1}M_{t,1}(b).
\]
By Lemma~\ref{lem:weighted-partial-fractions}, $M_{t,1}(b)=-(2b)^t /F(-b)\ne0$, a contradiction. Hence $-b\in A$ for every nonzero $b\in A$. The nonzero elements of $A$ therefore occur in opposite pairs; since $|A|$ is odd, $0\in A$, and hence $A=-A$.

Put $T=A\setminus\{0\}$. Then $|T|=2t$, $T=-T$, and $T\subseteq H$, and set
\[
 \Theta=\Phi_{T;d+t-1,t}-(-1)^t.
\]
For each $b\in T$, Proposition~\ref{prop:restricted-local-expansion}(i), applied to $T$ with $m=t-1$ and $\delta=1$, gives $[Y^j]\Theta(b+Y)=0$ for $0\le j<t-1$ while $[Y^{t-1}]\Theta(b+Y)\neq 0$ since $-b\in T$. Hence the order of vanishing of \(\Theta\) at \(b\) is exactly \(t-1\). In particular, $\Theta\neq 0$. Since $T\subseteq H$ and $(t-1)+t=|T|-1$, Lemma~\ref{lem:global-interpolation} shows that $\deg\Theta\le d$ and $X^{2t}\mid\Theta$. Thus $0$ is a root of $\Theta$ of multiplicity at least $2t$. As $0\notin T$, counting multiplicities gives $2t+2t(t-1)=2t^2\le \deg\Theta \le d$. On the other hand, the $t$ opposite pairs are exactly the unordered pairs with zero sum, so there are at most $\binom{2t+1}{2}-t=2t^2$ nonzero unordered sums. Hence $d\le2t^2$, and therefore $d=2t^2$. Equality also shows that all nonzero unordered sums are unique. This proves equation~\eqref{eq:restricted-zero-size} and completes the proof.
\end{proof}

Having established $A=-A$ and the auxiliary vanishing identity, we now extract the differential relation. In contrast with the previous subsection, we first isolate the $a=-b$ term in the partial-fraction identity and then expand the remaining expression as a formal power series in $Y$.

\begin{proposition}\label{prop:restricted-zero-global}
Assume $p>3$ and $A\rsum A=H\cup\{0\}$, and put $m=t-1+\varepsilon$. Then
\begin{equation}\label{eq:zero-global}
 X^2F''-(n-1+\varepsilon)XF'+\varepsilon nF=0.
\end{equation}
Consequently, for some $\xi\in\Fps$,
\begin{equation}\label{eq:zero-form}
 F(X)=X^n+\xi X^\varepsilon.
\end{equation}
If $m\ge2$, then for every $b\in A\setminus\{0\}$, writing $u=bF''(b)/F'(b)$ and $w=b^2F'''(b)/F'(b)$, one also has
\begin{align}
0={}&\binom{d+m}{2}-\binom m2+\frac{md}{m+1}(m-1-u) \notag\\
&\quad+\frac{d(d-1)}{(m+1)(m+2)}\left(-\binom m2+mu+\frac23w-u^2\right).
\label{eq:restricted-zero-third-local}
\end{align}
\end{proposition}

\begin{proof}
By Proposition~\ref{prop:restricted-zero-structure}, equation~\eqref{eq:restricted-zero-vanishing} holds, $A=-A$, and $m\ge1$. Fix $b\in A\setminus\{0\}$ and put
\[
 D=d+m,\qquad B=\binom Dm, \qquad u=b\frac{F''(b)}{F'(b)},\qquad w=b^2\frac{F'''(b)}{F'(b)}.
\]
Since $A=-A$, we have $F(-X)=(-1)^nF(X)$ and therefore $c_{-b}=(-1)^{n+1}c_b$. Put
\[
 U(Y)=\frac{F(b+Y)}{F'(b)Y}.
\]
Since $F(-b-Y)=(-1)^nF'(b)Y U(Y)$, Lemma~\ref{lem:weighted-partial-fractions} gives, for $0\le j \le n-1$,
\[
 \sum_{a\in A}c_a\frac{(b-a)^j}{b+a+Y}=\frac{c_{-b}(2b+Y)^j}{Y\,U(Y)}.
\]
The $a=-b$ term on the left is $c_{-b}(2b)^j/Y$. Subtracting this term from both sides gives
\[
 \frac{c_{-b}(2b+Y)^j}{Y\,U(Y)}-\frac{c_{-b}(2b)^j}{Y}\equiv M_{j,1}(b)-M_{j,2}(b)Y \pmod{Y^2}.
\]
Here
\[
 U(Y)\equiv 1+\frac{u}{2b}Y+\frac{w}{6b^2}Y^2 \pmod{Y^3}.
\]
Using Lemma~\ref{lem:coefficient-expansions} to expand $U(Y)^{-1}$ therefore gives
\begin{align*}
 M_{j,1}(b)=c_{-b}(2b)^{j-1}(j-u), \qquad M_{j,2}(b)=c_{-b}(2b)^{j-2} \left(-\binom j2+ju+\frac23w-u^2\right).
\end{align*}
Apply Proposition~\ref{prop:restricted-local-expansion}(i) with $\delta=0$. The associated polynomial $\Psi=\Phi_{A;d+m,m}-\varepsilon(-1)^m$ is identically zero by equation~\eqref{eq:restricted-zero-vanishing}, and part~(i) gives $(2b)^d-1=(-1)^{n+1}B$. Since $A=-A$, we have $-b\in A$. By equation~\eqref{eq:restricted-zero-size}, $d\ge2$, so $1<d$ and part~(ii) applies with $\delta=0$ and $r=1$. Since $\Psi=0$, the left-hand side of equation~\eqref{eq:restricted-coefficient-reduction} vanishes. Substituting the preceding relation and the formula for $M_{m,1}(b)$ into the right-hand side, and dividing by the common nonzero factor $Bc_{-b}(2b)^{m-1}$, gives $0=d+\frac d{m+1}(m-u)$, so
\[
 b\frac{F''(b)}{F'(b)}=u=2m+1=n-1+\varepsilon.
\]
If $m\ge2$, then equation~\eqref{eq:restricted-zero-size} gives $d>2$, so $2<d$ and Proposition~\ref{prop:restricted-local-expansion}(ii) also applies with $\delta=0$ and $r=2$. Since $\Psi=0$, equation~\eqref{eq:restricted-coefficient-reduction} reads
\[
\begin{aligned}
0=\binom D2 c_b(2b)^{m-2}\bigl((2b)^d-1\bigr)-B\binom m2 c_{-b}(2b)^{m-2}+m\binom D{m+1}M_{m-1,1}(b)+\binom D{m+2}M_{m,2}(b).
\end{aligned}
\]
Using $(2b)^d-1=(-1)^{n+1}B$, $c_{-b}=(-1)^{n+1}c_b$, and the formulas for $M_{m-1,1}(b)$ and $M_{m,2}(b)$, and then dividing by $Bc_{-b}(2b)^{m-2}$, gives
\[
 0=\binom D2-\binom m2+\frac{md}{m+1}(m-1-u)+\frac{d(d-1)}{(m+1)(m+2)}\left(-\binom m2+mu+\frac23w-u^2\right),
\]
which is equation~\eqref{eq:restricted-zero-third-local}.

Thus every nonzero root of $F$ is a root of
\[
 Q(X)=X^2F''-(n-1+\varepsilon)XF'+\varepsilon nF.
\]
Since $A=-A$, the root $0$ belongs to $A$ exactly when $\varepsilon=1$; in that case $Q(0)=nF(0)=0$ as well. Hence $Q$ vanishes at all $n$ distinct roots of $F$. The coefficient of $X^n$ in $Q$ is $n(n-1)-n(n-1+\varepsilon)+\varepsilon n=0$, so $\deg Q\le n-1$. Therefore $Q=0$, proving equation~\eqref{eq:zero-global}.

For $0\le j\le n$, if $a_j=[X^j]F$, equation~\eqref{eq:zero-global} gives $(j-\varepsilon)(j-n)a_j=0$.
As noted at the beginning of the section, $n<p$. Thus, for $0\le j\le n$, the scalar $(j-\varepsilon)(j-n)$ can vanish in $\Fp$ only when $j=\varepsilon$ or $j=n$. Since $F$ is monic and squarefree, equation~\eqref{eq:zero-form} follows with $\xi\ne0$.
\end{proof}

\begin{proof}[Proof of Theorem~\ref{thm:restricted-zero}]
If $|A|<2$, then $A\rsum A=\varnothing$; if $|A|=2$, then $A\rsum A$ consists of a single element. If $p=3$ and $A=\mathbb F_3$, then $\mathbb F_3\rsum\mathbb F_3=\mathbb F_3$, so no subset $A\subseteq\mathbb F_3$ satisfies $A\rsum A=\{0,1\}$. Hence next we assume $p>3$.
By Proposition~\ref{prop:restricted-zero-structure}, $d=\binom n2-t=2t(t-1+\varepsilon)$, while Proposition~\ref{prop:restricted-zero-global} gives $F(X)=X^n+\xi X^\varepsilon$ for some $\xi\in\Fps$.

\emph{Even $n$.} Let $n=2t$. Then $d=2t(t-1)$ and $t\ge2$.
Equation~\eqref{eq:zero-form} gives $F(X)=X^{2t}+\xi$. Choosing $\gamma\in A$, the $2t$ roots are $\gamma K$, where $K<\Fps$ is the subgroup of order $2t$. Thus $A=\gamma K$. Multiplication by any $k\in K$ fixes $A$ and hence fixes $A\rsum A=H\cup\{0\}$. Therefore $kH=H$, so $k\in H$; thus $K\subseteq H$. Suppose $t\ge3$. Since $|K|=2t\ge6$, we may choose $u\in K$ with $u^4\ne1$. Also $-1\in K$ because $|K|$ is even. Since $u^4\ne1$, the three nonzero restricted sums $\gamma(1+u)$, $\gamma(1-u)$, and $\gamma(1-u^2)$ belong to $H$. As $H$ is a multiplicative subgroup,
\[
 \gamma=\frac{\gamma(1+u)\,\gamma(1-u)}{\gamma(1-u^2)}\in H.
\]
Thus $A=\gamma K\subseteq H$. Set $A'=A\cup\{0\}$. Since $A\subseteq H$, adjoining $0$ creates no new nonzero restricted sums, so $A'\rsum A'=H\cup\{0\}$. As $|A'|=2t+1$, equation~\eqref{eq:restricted-zero-size} applied to $A'$ gives $d=2t^2$, contradicting $d=2t(t-1)$. Hence $t=2$ and $(n,d)=(4,4)$.

\emph{Odd $n$.} Let $n=2t+1$, so $d=2t^2$. The case $t=1$ gives $(n,d)=(3,2)$; assume $t\ge2$. Equation~\eqref{eq:zero-form} gives $F(X)=X^{2t+1}+\xi X$. Thus $0\in A$, and choosing a nonzero root $\gamma$ shows that the remaining $2t$ roots are $\gamma K$, where $K<\Fps$ is the subgroup of order $2t$. Thus $A=\{0\}\cup\gamma K$. Since $0\in A$, every element of $\gamma K$ occurs as the restricted sum $0+a$, so $\gamma K\subseteq H$. In particular $\gamma\in H$, and hence $K\subseteq H$. Multiplication by $\gamma^{-1}\in H$ preserves $H\cup\{0\}$, so after scaling we may assume
\[
 A=\{0\}\cup K,\qquad F(X)=X^{2t+1}-X.
\]

Write $r=(p-1)/d\ge2$. Since $F(X)=X^{2t+1}-X$, we have
\[
u:=\frac{F''(1)}{F'(1)}=2t+1, \qquad w:=\frac{F'''(1)}{F'(1)}=(2t+1)(2t-1).
\]
Since $m=t$, substituting these values into equation~\eqref{eq:restricted-zero-third-local} and simplifying gives
\[
 (2t+1)(t-2)d\equiv(5t-2)(t+1)\pmod p.
\]
For $t=2$ this gives $24\equiv0\pmod p$, impossible since $p\ge17$; hence $t\ge3$.

We now lift this congruence to an integer identity. Since $d=2t^2$, set $D=2t^2(2t+1)(t-2)-(5t-2)(t+1)$. The congruence above says that $p\mid D$. For $t\ge3$ we also have
\[
 D=(t-1)(t-2)+2t^2(t+1)(2t-5)>0,
\]
while $(2t+1)(t-2)<2t^2$ gives $D<4t^4$. Moreover, $D\equiv(t-1)(t-2)\pmod{2t^2}.$ Write $D=kp$. Since $p=rd+1=2rt^2+1$ with $r\ge2$, we have $p\ge4t^2+1$ and $p\equiv1\pmod{2t^2}$. Therefore
\[
 0<k<t^2, \qquad k\equiv D\equiv(t-1)(t-2)\pmod{2t^2}.
\]
Both $k$ and $(t-1)(t-2)$ lie in $[0,t^2)$, so the congruence forces $k=(t-1)(t-2).$ Substituting $D=kp$ and $p=2rt^2+1$ now gives $r(t-1)(t-2)=(t+1)(2t-5).$ Since $3(t-1)(t-2)-(t+1)(2t-5)=(t-3)^2+2>0,$ the right-hand side is strictly less than $3(t-1)(t-2)$. Thus $2\le r<3$, so $r=2$, and the preceding relation then gives $t=3$.

Now $|K|=6$. Let $\zeta\in K$ have order $3$. Since $1+\zeta+\zeta^2=0$ and $-\zeta^2\in K$, the nonzero element $1+\zeta$ has the two distinct unordered representations $1+\zeta=0+(-\zeta^2)$, contradicting Proposition~\ref{prop:restricted-zero-structure}. Hence the case $t\ge2$ is impossible. The odd case is therefore only $(n,d)=(3,2)$, and altogether $d\in\{2,4\}$.
\end{proof}

\begin{remark}\label{rem:restricted-zero-classification}
The proof again gives the full classification. If $|H|=2$, then $(n,d)=(3,2)$ and $A=\{0,a,-a\}$ for some $a\ne0$. Since $0+a=a\in H=\{1,-1\}$, we obtain $A=\{0,1,-1\}$. If $|H|=4$, then $(n,d)=(4,4)$ and equation~\eqref{eq:zero-form} gives $F(X)=X^4+\xi$. Choosing $\gamma\in A$, its roots are $\gamma K$ for a subgroup $K$ of order $4$. Since $K$ stabilizes $A\rsum A=H\cup\{0\}$, we have $K\subseteq H$ and hence $K=H$. If $i\in H$ has order $4$, then $\gamma(1+i)\in H$, so
\[
 A=\frac{1}{1+i}H =\left\{\pm\frac{1+i}{2},\ \pm\frac{1-i}{2}\right\}.
\]
Conversely, both sets give $H\cup\{0\}=A\rsum A$.
\end{remark}

\section*{Acknowledgments}
The authors thank Ernie Croot, Alexander Kalmynin, Giorgis Petridis, Misha Rudnev, Ilya Shkredov, J\'ozsef Solymosi, and Craig Spencer for helpful discussions. S.~Yoo was supported by the Institute for Basic Science (IBS-R029-C1). The authors also acknowledge the use of ChatGPT by OpenAI to assist with the writing of the manuscript and checking algebraic computations. All proofs and their mathematical ideas were given by the authors.

\bibliographystyle{abbrv}
\bibliography{references}

\end{document}